\documentclass[12pt, a4paper, parskip=half, abstracton]{scrartcl}

\usepackage{array}
\usepackage{marginnote}
\usepackage{xcolor}
\usepackage{amscd,amssymb,amsfonts,amsmath,latexsym,amsthm}
\usepackage{hyperref}
\usepackage[all,cmtip]{xy}
\usepackage{mathrsfs}
\usepackage{graphicx}
\usepackage{bm} 

\usepackage{etoolbox}

\def\hB{\hspace*{\fill}$\qed$}

\usepackage[nottoc]{tocbibind} 

\usepackage{defs_pp1}
\usepackage{slashed}
\usepackage[utf8]{inputenc}
\usepackage{microtype}
\usepackage[english]{babel}
\usepackage{mathtools}
\usepackage{bm}
\usepackage{esvect}

\title{Non-unital $C^{*}$-categories, (co)limits, crossed products and exactness  }
\author{
Ulrich Bunke\thanks{Fakult{\"a}t f{\"u}r Mathematik,
Universit{\"a}t Regensburg,
93040 Regensburg,
ulrich.bunke@mathematik.uni-regensburg.de} 
}

\numberwithin{equation}{section}
\newtheorem{theorem}{Theorem}[section] 
\newtheorem{prop}[theorem]{Proposition}
\newtheorem{lem}[theorem]{Lemma}

\newtheorem{ddd}[theorem]{Definition}
\newtheorem{kor}[theorem]{Corollary}

\theoremstyle{remark}
\theoremstyle{definition}

\newtheorem{rem}[theorem]{Remark}

\newcommand{ \Ker }{\mathrm{Ker}}

\newcommand{\F}{\mathbb{F}}

\newcommand{\bQ}{\mathbf{Q}}

\newcommand{\Ob}{\mathrm{Ob}}

\newcommand{\bB}{{\mathbf{B}}}

\newcommand{\incl}{\mathrm{incl}}

\newcommand{\Groupoids}{\mathbf{Groupoids}}

\newcommand{\bA}{{\mathbf{A}}}

\newcommand{\bK}{{\mathbf{K}}}

\newcommand{\cD}{{\mathcal{D}}}

 \newcommand{\Cat}{{\mathbf{Cat}}}

\newcommand{\Ccat}{{\mathbf{C}^{\ast}\mathbf{Cat}}}
\newcommand{\Calg}{{\mathbf{C}^{\ast}\mathbf{Alg}}}

\newcommand{\op}{\mathrm{op}}

\newcommand{\codom}{\mathrm{codom}}

\begin{document}
 	
\maketitle

\begin{abstract}
We provide a reference for basic categorial properties of the categories of (possibly non-unital) $\C$-linear $*$-categories or $C^{*}$-categories, and (not necessarily unit-preserving) functors.  Generalizing the classical case of algebras with $G$-action, 
we   extend the  construction of   crossed products to categories with $G$-action.  We will show that the crossed product functor preserves exact sequences and excisive squares and sends weak equivalences to equivalences.
 \end{abstract}

\tableofcontents
\setcounter{tocdepth}{5}


\newcommand{\nCcat}{C^{*}\mathbf{Cat}^{\mathrm{nu}}}
\renewcommand{\Ccat}{C^{*}\mathbf{Cat}}
\newcommand{\alg}{\mathrm{alg}}
\newcommand{\nClincat}{\mathbf{{}^{*}\Cat^{\mathrm{nu}}_{\C}}}
\newcommand{\npClincat}{\mathbf{{}_{\mathrm{pre}}^{*}\Cat^{\mathrm{nu}}_{\C}}}
\newcommand{\pClincat}{\mathbf{{}_{\mathrm{pre}}^{*}\Cat_{\C}}}

\newcommand{\Clincat}{\mathbf{{}^{*}\Cat_{\C}}}
\renewcommand{\Calg}{C^{*}\mathbf{Alg}}
\newcommand{\nCalg}{C^{*}\mathbf{Alg}^{\mathbf{nu}}}
\renewcommand{\nCalg}{C^{*}\mathbf{Alg}^{\mathrm{nu}}}

\newcommand{\Algc}{{}^{*}\mathbf{Alg}_{\mathbb{C}}}
\newcommand{\nAlgc}{{}^{*}\mathbf{Alg}^{\mathrm{nu}}_{\mathbb{C}}}
\newcommand{\npAlgc}{{}_{\mathrm{pre}}^{*}\mathbf{Alg}^{\mathrm{nu}}_{\mathbb{C}}}
\newcommand{\pAlgc}{{}_{\mathrm{pre}}^{*}\mathbf{Alg}_{\mathbb{C}}}
\newcommand{\Bd}{\mathrm{Bd}}
\newcommand{\Compl}{\mathrm{Compl}}
\newcommand{\inj}{\mathrm{inj}}
\newcommand{\nClincatinj}{{}^{*}\mathbf{Cat}^{\mathrm{nu}}_{\C,\inj}}
\newcommand{\nCcatinj}{C^{*}\mathbf{Cat}^{\mathrm{nu}}_{\inj}}

\section{Introduction}

If a group $G$ acts on a  (not necessarily commutative or unital) ring $A$ by automorphisms, then we can construct  in a functorial way a new ring  $A\rtimes^{\alg}G$ called the crossed product of $A$ with $G$.   Its underlying abelian group   is given by $\bigoplus_{g\in G} A$. Let $(a,g)$ denote the element of $A\rtimes^{\alg}G$  corresponding to  the element  $a$ of $A$ in the summand with index $g$ in $G$. 
Then the multiplication in  the crossed product is determined by bi-linearity and the rule 
$(a',g')(a,g)=((g^{-1}a')a,g'g)$, where $ga'$ denotes the image of $a'$ under the automorphism of $A$ given by $g$.

A $*$-algebra over $\C$ is an algebra $A$ over $\C$ with a  complex anti-linear involution $a\mapsto a^{*} $ such that $(a'a)^{*}= a^{*}a^{\prime*}$.
If $A$ is a $*$-algebra over $\C$ and $G$ acts by automorphisms of $*$-algebras, then $A\rtimes^{\alg} G$ is again   a   $*$-algebra over $\C$ with   involution determined by  $(a,g)^{*}=(ga^{*},g^{-1})$. 

A $C^{*}$-algebra is a $*$-algebra $A$ over $\C$ which is complete with respect to some\footnote{The norm is actually unique.} norm $\|-\|_{A}$ satisfying $\|a^{*}\|_{A}=\|a\|_{A}$ for all $a$ in $A$, $\|aa'\|_{A}\le \|a\|_{A}\|a'\|_{A}$  for all $a,a'$ in $A$, and the $C^{*}$-condition $\|a^{*}a\|_{A}=\|a\|_{A}^{2}$ for all $a$ in $A$. If $A$ is a $C^{*}$-algebra
with $G$-action, then a $C^{*}$-algebraic crossed product  $A\rtimes G$  is obtained from $A\rtimes^{\alg}G$ by completion with respect  to a suitable $C^{*}$-norm. In general there are various interesting choices of this norm. For the purpose of the present paper we consider the maximal norm $\|-\|_{\max}$ on $A\rtimes^{\alg}G$ defined by
$$\|x\|_{\max}:=\sup_{\rho} \|\rho(x)\|_{B}\ ,$$ where
$\rho$ runs over all homomorphisms $\rho:A\rtimes^{\alg}G\to B$ of $*$-algebras  over $\C$  with target a $C^{*}$-algebra (note the discussion after the Corollary \ref{qergiowegwregergewgrg}). 
    
The crossed product is often considered as a kind of homotopy quotient of the ring $A$ by the group action. One of the outcomes of the present paper is to make this idea precise in a technical sense at least in the unital case. To this end 
we embed  the category of $*$-algebras over $\C$ into the category of $\C$-linear $*$-categories,  and the category of 
$C^{*}$-algebras into   the category of $C^{*}$-categories.  Before we can give the precise formulation in Corollary \ref{wthpowkphtrhwegrwegwregw} 
we will introduce the basic notions which go into its statement.

A $\C$-linear $*$-category is a small (possibly non-unital) category  which is enriched in $\C$-vector spaces,  and which has  an involution $*$ which fixes the objects, reverses arrows, and which acts anti-linearly on the $\Hom$-vector spaces.  The algebra of endomorphisms of every object in such a category  then becomes a $*$-algebra over $\C$. 

A morphism between    $\C$-linear $*$-categories  is a (not necessarily unit-preserving)  functor which is compatible with the enrichment and the involutions. In this way we obtain the category $\nClincat$ of    small $\C$-linear $*$-categories and     functors. The superscript $\mathrm{nu}$ stands for non-unital and  indicates that we do not require the existence of units or that functors preserve units.
The category of  $*$-algebras $\nAlgc$ over $\C$  embeds into $\nClincat$ as the full subcategory of  $\C$-linear $*$-categories with a single object.

The relation between $C^{*}$-categories and $\C$-linear $*$-categories is similar as in the case of algebras. If $\bC$ is a $\C$-linear $*$-category, then we can define a maximal semi-norm (which might assume the value $\infty$) on the morphism spaces by
$$\|f\|_{\max}:=\sup_{\rho} \|\rho(f)\|_{B}\ ,$$
where $\rho$ runs over all morphisms of $\C$-linear $*$-categories $\rho:\bC\to B$
with target a $C^{*}$-algebra (considered as a  $\C$-linear $*$-category with a single object). 

A   $\C$-linear $*$-category  is a   $C^{*}$-category if its maximal semi-norm  is a  finite norm,   and if the morphism spaces 
are complete with respect to this norm. The $*$-algebra of endomorphisms of an object in a $C^{*}$-category is a $C^{*}$-algebra.
We refer to \cite[Rem 2.15]{startcats} for a discussion of the equivalence of this definition with other (previous) definitions 
in the literature.\footnote{The word ``parallel'' must be deleted in  Condition 4. in  \cite[Rem 2.15]{startcats} and also in point 4.  in the text after  \cite[Def. 9]{startcats}.}

A  morphism between $C^{*}$-categories is just a   morphism between $\C$-linear $*$-categories.
In this way we can consider the category $\nCcat$ of    small $C^{*}$-categories  
as a full subcategory of $\nClincat$.   Moreover, the category of $C^{*}$
algebras $\nCalg$ embeds into as the full subcategory of $\nCcat$ consisting of $C^{*}$-categories with a single object.
 
We let $\Clincat$ be the subcategory of $\nClincat$ of small unital $\C$-linear $*$-categories and unital functors.  Then the category 
$\Ccat:=\Clincat\cap \nCcat$ is the category of unital  $C^{*}$-categories and unital functors. Furthermore,
$\Algc:=\nAlgc\cap \Clincat$ is the category of unital $*$-algebras over $\C$ and unital homomorphisms, and finally $\Calg:=\nCalg\cap \Ccat$ is the category of unital $C^{*}$-algebras and unital homomorphisms.
 
It is known \cite{DellAmbrogio:2010aa}, \cite[Thm. 8.1]{startcats} that the categories $\Clincat$ and $\Ccat$ are complete and cocomplete, i.e., that they admit limits and colimits for all diagrams indexed by small categories. 
Since we are going  perform  categorical constructions  in the non-unital cases the following is useful to know.

\begin{theorem}[Theorem \ref{riguhqwieufqewfeqfqewf}]
The categories $\nClincat$ and $\nCcat$ are complete and cocomplete.
\end{theorem}

 For a group $G$ we let $\Fun(BG,\cC)$ denote the category of objects with $G$-action and equivariant morphisms in a category $\cC$.    The main construction of the present paper is the extension of the crossed product functors
 $$-\rtimes^{\alg}G: \Fun(BG,\nAlgc)\to \nAlgc  \ , \quad -\rtimes G:\Fun(BG,\nCalg)\to \nCalg $$
 described above to   categories, i.e. we will extend these functors 
  to  functors
$$-\rtimes^{\alg}G:\Fun(BG,\nClincat)\to \nClincat\ , \quad  -\rtimes G:\Fun(BG,\nCcat)\to \nCcat$$
(see Definitions \ref{wergiojwergregregwergerg} and \ref{feivoevvewfvfevfdsv}).
 Both versions of the crossed product functors preserve unitality.
 
  The restriction of the definition of the crossed product from $C^{*}$-categories to $C^{*}$-algebras (considered as $C^{*}$-categories with a single object) differs slightly from the standard Definition \ref{qreugiqwgrefwqewfqewf} for $C^{*}$-algebras. In Proposition \ref{egwergergwrgw} we verify that both definitions provide the same result.

For unital $\C$-linear $*$-categories or unital $C^{*}$-categories we have the notion of a unitary  isomorphism between morphisms \cite[Def. 5.1]{startcats}. Morphisms which are invertible up to unitary isomorphisms are called unitary equivalences   \cite[Eef. 2.4]{DellAmbrogio:2010aa}, \cite[Def. 5.2]{startcats}. We will use the symbol  $\approx$ in order to denote the relation of unitary equivalence between objects. 

 Forming the Dwyer-Kan localization of
$\Clincat$ or $\Ccat$ with respect to  collection  of unitary equivalences we obtain $\infty$-categories $\Clincat_{\infty}$ and $\Ccat_{\infty}$ (see Definition \ref{erkgowergerwgwergergwerg}, \cite[Def. 5.7]{startcats}) which model a homotopy theory of   unital $\C$-linear $*$-categories, or  unital $C^{*}$-categories, respectively.
For $\bC$ in $\Fun(BG,\Clincat)$ or $\Fun(BG,\Ccat)$
we get  a notion of  a homotopy quotient  of $\bC$ by $G$  denoted by $\bC_{hG}$\footnote{In the notation of Theorem \ref{weiogwgwerrewgwregwreg} we have 
$\bC_{hG}:=\colim_{BG }\ell_{BG}^{\alg}(\bC)$ or $\bC_{hG}:=\colim_{BG }\ell_{BG}(\bC)$, respectively.}. The homotopy quotient $\bC_{hG}$ is an object of $\Clincat$, or $\Ccat$ respectively,  which is well-defined up to unitary equivalence.

The following  is a reformulation of  Theorem \ref{weiogwgwerrewgwregwreg}.
\begin{theorem}\label{ioqjeroijgwregwergregwg} \mbox{}
\begin{enumerate} \item  If $\bC$ is a  unital $\C$-linear $*$-category   with $G$-action, then 
  $\bC\rtimes^{\alg}G\approx \bC_{hG}$ .  
  \item If $\bC$ is a unital $C^{*}$-category  with $G$-action, then 
  $\bC\rtimes G\approx \bC_{hG}$. 
\end{enumerate}
\end{theorem}

 If $A$ and $B$ are $*$-algebras over $\C$ or $C^{*}$-algebras,  then the relation  $A\approx B$ implies $A\cong B$. 
In the following corollary $A_{hG}$ is still interpreted in the respective category of $*$-categories.

\begin{kor}\label{wthpowkphtrhwegrwegwregw}\mbox{}
\begin{enumerate} \item  If $A$ is a  unital  $*$-algebra over $\C$  with $G$-action, then  $A\rtimes^{\alg}G$ is the unique (up to isomorphism) unital  $*$-algebra over $\C$ which is unitarily equivalent to $A_{hG}$.
  \item If $A$ is a  unital $C^{*}$-algebra    with $G$-action, then  $A\rtimes G$ is the unique (up to isomorphism) unital  $C^{*}$-algebra   which is unitarily equivalent to $A_{hG}$.
\end{enumerate}
\end{kor}

We now consider invariance properties of the crossed products.
 Assume that $\phi:\bC\to \bC^{\prime}$ is a morphism in $\Fun(BG,\Clincat)$.  We first consider the obvious case that  
  that there exists an inverse equivalence $\psi:\bC^{\prime}\to \bC $ in $\Fun(BG,\Clincat)$, i.e, 
 the compositions $\psi\circ \phi$ and $\phi\circ \psi$ are unitarily isomorphic  to the respective identities in  $\Fun(BG,\Clincat)$.  Then $\phi\rtimes G:\bC\rtimes^{\alg} G\to \bC'\rtimes^{\alg} G$ and $\psi\rtimes G: \bC'\rtimes^{\alg} G\to \bC\rtimes^{\alg} G$ are
inverse to each other (up to unitary isomorphism)  unitary  equivalences in $\nClincat$.
An analogous statement holds true in the $C^{*}$-case.

  Theorem \ref{ioqjeroijgwregwergregwg}  implies that
the crossed product preserves a weaker form of equivalences.
A morphism $\phi:\bC\to \bC^{\prime}$  in $\Fun(BG,\Clincat)$
or $\Fun(BG,\Ccat)$ is called a weak equivalence (see Definition \ref{wtihowergergwegergwreg}) if it becomes a unitary equivalence after forgetting the $G$-action. Thus for $\phi$ being a weak equivalence we drop the requirement that the inverse equivalence $\psi$  is equivariant.  
The following is Proposition \ref{efiobgebgewrverbvewvbev}.
   \begin{prop} [crossed product sends weak equivalences to equivalences]\mbox{}\label{efiobgebgewrverbvewvbev1}
\begin{enumerate}
\item \label{wthiorthrherherthetrhtrh}
If $ \phi:\bC\to \bD$  is a weak equivalence in $\Fun(BG,\Clincat)$, then the induced morphism
$\phi\rtimes^{\alg}G:\bC\rtimes^{\alg}G\to \bD\rtimes^{\alg}G$
is a unitary equivalence.
\item  \label{wthiorthrherherthetrhtrh1}
If $\phi: \bC\to \bD$  is a weak equivalence in $\Fun(BG,\Ccat)$, then  the induced morphism
$\phi\rtimes G:\bC\rtimes G\to \bD\rtimes G $
is a unitary equivalence. \end{enumerate}
\end{prop}

In the non-unital case we still have a precise relation of the crossed product with a colimit in $\nClincat$ or $\nCcat$. We refer to Proposition  \ref{regiueqrhgiqwgewgqgfewrg} for the statement.

The notion of an exact sequence of $*$-algebras over $\C$ or $C^{*}$-algebras has a natural generalization Definition \ref{ergiowejogergergwergwergwregw} to the case of categories. 
A sequence with $G$-actions is exact if it becomes an exact sequence after forgetting the $G$-action.
It is essentially obvious from the definition that the algebraic crossed product
$$-\rtimes G:\Fun(BG, \nClincat)\to  \nClincat$$ preserves exact sequences 
(this fact is stated as Theorem \ref{fbgbgrbsfbsdfbs}.\ref{sivijaoddsvadsv}).  
Because of the completions involved in its construction,  it is not so obvious  but well known, that the $C^{*}$-crossed product preserves exact sequences of $C^{*}$-algebras\footnote{Note that we define the crossed product with the maximal norm.}. The following  Theorem  \ref{fbgbgrbsfbsdfbs}.\ref{giowejgoeregwegergrg}
 extends this assertion to $C^{*}$-categories.

\begin{theorem}[exactness of crossed product] If $$0\to \bC\to \bD\to \bQ\to 0$$ is an exact sequence in $\Fun(BG,\nCcat)$ such that $\bD$ is unital, then 
$$0\to \bC\rtimes G\to \bD\rtimes G\to \bQ\rtimes G\to 0$$
is an exact sequence in $\nCcat$.
\end{theorem}

One basic motivation for the present paper is to provide a  reference for constructions with $C^{*}$-categories which go into the construction of a version of equivariant coarse $K$-homology in \cite{coarsek}. 
The non-equivariant case has beed worked out in \cite[Sec. 8]{buen}. 
The proof of excision in \cite{coarsek}  uses the notion of excisive squares of $C^{*}$-categories.  
  This  notion  is relevant since the topological $K$-theory functor for $C^{*}$-categories (see Definition \ref{qrgijewofweqfqweqwefqefq}) sends excisive squares of $C^{*}$-categories to push-out squares of spectra (see Proposition \ref{ergiuheigerwgqwfwefqwef}).

%
    \begin{ddd} A commutative square \begin{equation}\label{asdv2e4fwdwdqwdqwdqdqwqev}
\xymatrix{\bA\ar[r]\ar[d]&\bB\ar[d]\\\bC\ar[r]&\bD}
\end{equation}
 in $\nCcat$ is called excisive, if: 
    \begin{enumerate}
  \item $\bB$ and $\bD$ are unital and the morphism $\bB\to \bD$ is unital.
  \item The morphisms $\bA\to \bB$ and $\bC\to \bD$ are inclusion of ideals. 
  \item  The induced morphism between the quotients  $\bB/\bA\to  \bD/\bC$  is a unitary  equivalence. 
  \end{enumerate}
  \end{ddd}

  A square of the shape \eqref{asdv2e4fwdwdqwdqwdqdqwqev} in $\Fun(BG,\nCcat)$ is excisive, if it becomes excisive after forgetting the $G$-action. 
  
  \begin{theorem}[crossed product preserves excisive squares]
  If the square of the shape  \eqref{asdv2e4fwdwdqwdqwdqdqwqev} is an excisive square  in  $\Fun(BG,\nCcat)$, then  \begin{equation}\label{asdv2e4fwdqwfefefwefwefwefwefeffqev}
\xymatrix{\bA\rtimes G\ar[r]\ar[d]&\bB\rtimes G\ar[d]\\\bC\rtimes G\ar[r]&\bD\rtimes G}
\end{equation}
is an excisive square in $\nCcat$.
 \end{theorem}

Besides proving the results stated so far, in Sections \ref{qriugoqrgqfewfqwef} and  \ref{eqwrughuqiorgfwefwqefqwefqewf} we provide a reference for various facts about the categories introduced above. We discuss adjunctions relating the unital and the non-unital cases. Furthermore we provide adjunctions which relate $\C$-linear $*$-categories with $C^{*}$-categories via the intermediate category of pre-$C^{*}$-categories.  The unital case of all this has been worked out in 
\cite[Sec. 3]{startcats}, and in this paper we provide the non-unital generalizations.

{\em Acknowledgement: The author  has benefitted   from the cooperations with Alexander Engel. Some of the  ideas used in the present paper have been developed  in the joined project \cite{buen}.  
The author furthermore thanks Siegfried Echterhoff for an encouraging discussion.
Finally, the author was supported by the SFB 1085 (Higher Invariants) founded by the DFG.}

\section{Unital and non-unital $\C$-linear $*$-categories}\label{qriugoqrgqfewfqwef}

In order to fix set-theoretic size issues we fix a sequence of two Grothendieck universes whose elements will be called small and large sets.

A $\C$-linear $*$-category is a category that is enriched over $\C$-vector spaces and is equipped
with an involution which fixes objects, and which acts anti-linearly on the morphism vector spaces reversing their direction 
\cite[Def. 2.3]{startcats}. A morphism between $\C$-linear $*$-categories is a functor which is compatible with the enrichment and which preserves the involution. The large category of small $\C$-linear $*$-categories   will be denoted by $\Clincat$.

If we omit the requirement that a category has identity morphisms, and that functors preserve identities, then we arrive at the notions of a possibly non-unital $\C$-linear $*$-category and of a  possibly non-identity preserving morphisms. We let 
$\nClincat$ denote the large category of  possibly non-unital  small  $\C$-linear $*$-categories  and  possibly non-identity preserving  morphisms. 
 We have an inclusion functor
 \begin{equation}\label{ewfqwfeiuhqwiefqewfqewf}\incl: \Clincat\to \nClincat\ .
  \end{equation} 

\begin{prop}
The inclusion functor \eqref{ewfqwfeiuhqwiefqewfqewf} is the left- and right adjoint of adjunctions
\begin{equation}\label{egrggqwwefeqwfr4r41}  (-)^{+}:\nClincat\leftrightarrows \Clincat:\incl\ ,  \end{equation} 
and
 \begin{equation}\label{qewfoihqoiewfqwfqwefqewf}
\incl:\Clincat\leftrightarrows \nClincat:U\ .
\end{equation}
\end{prop}
\begin{proof}
 The functor $(-)^{+}$ is the unitalization functor. Let $\bC$ be in $\nClincat$. 
Its  unitalization     $\bC^{+}$  has the following description:
\begin{enumerate}
\item objects: $\bC^{+}$ has the same set of objects as $\bC$.  \item
morphisms: The $\C$-vector space of morphisms in $\bC^{+}$ between two objects $C,C^{\prime}$ in $\bC$ is given by
 $$\Hom_{\bC^{+}}(C,C^{\prime}):=\left\{
 \begin{array}{cc} \Hom_{\bC}(C,C^{\prime})&C\not=C^{\prime}\\ \Hom_{\bC}(C,C )\oplus \C &C=C^{\prime}
 \end{array}\right.$$ 
 \item involution:
The involution sends a morphism $f:C\to C^{\prime}$ in $\bC^{+}$ to $f^{*}:C'\to C$ if $C\not=C^{\prime}$, and
the morphism $(f,\lambda):C\to C$  in $\bC^{+}$ to $(f^{*},\bar \lambda)$. 
\item composition:
The composition is determined by the following cases and the compatibility with the involution.
\begin{enumerate} 
\item 
If $C,C',C''$ are three distinct objects of $\bC$, and $f:C\to C'$ and $f':C'\to C''$ are morphisms in $\bC^{+}$, then their composition is given by $f'\circ f:C\to C''$. 
\item If $C\not=C'$ and $f:C\to C'$ and $(f',\lambda):C'\to C'$ are morphisms in $\bC^{+}$, then  
  their composition is given by 
  $(f',\lambda)\circ f:=(f'\circ f+\lambda f):C\to C^{\prime}$.
  \item Finally, if $(f,\lambda),(f',\lambda'):C\to C$ are two endomorphisms of $C$  in $\bC^{+}$, then  $(f',\lambda')\circ (f,\lambda):=(f'\circ f+\lambda' f+f' \lambda,\lambda' \lambda):C\to C$.\end{enumerate}
  \end{enumerate}
  If $\phi:\bC\to \bC'$ is a morphism in  $\nClincat$, then we define
  $\phi^{+}:\bC^{+}\to \bC^{\prime,+}$ as follows:
  \begin{enumerate}
  \item objects: $\phi^{+}$  acts on objects as $\phi$.
  \item morphisms: If $f$ or $(f,\lambda)$ is a morphism in $\bC^{+}$, then its image under $\phi^{+}$ is given by
  $\phi(f)$ or $(\phi(f),\lambda)$, respectively.
  \end{enumerate}
This finishes the description of the unitalization functor.
  
  The unit of the adjunction \eqref{egrggqwwefeqwfr4r41} is given by the family $(\alpha_{\bC})_{\bC\in \nClincat}$  of  morphisms 
  $$\alpha_{\bC}:\bC\to \incl(\bC^{+})\ .$$ Here $\alpha_{\bC}$   is the identity on objects and sends a morphism
  $f:C\to C'$ in $\bC$ to the morphism $f$ in $\bC^{+}$ if $C\not=C^{\prime}$, or to  the morphism $(f,0)$ in $\bC^{+}$ if $C=C^{\prime}$.
 
For $\bC$ in $\nClincat$ and $\bD$ in $\Clincat$ we consider the map \begin{equation}\label{ergoih1o5gweg}
 \Hom_{\Clincat}(\bC^{+},\bD)\to   \Hom_{\nClincat}(\bC,\incl(\bD))
\end{equation}
  which sends 
 $\phi:\bC^{+}\to \bD$ to  the composition
$$\bC\stackrel{\alpha_{\bC}}{\to}  \incl(\bC^{+})\stackrel{\incl(\phi)}{\to} \incl(\bD)\ .$$
It is straightforward to check that    \eqref{ergoih1o5gweg} is a bijection   and bi-natural in $\bC$ and $\bD$. This finishes the description of the adjunction  \eqref{egrggqwwefeqwfr4r41}.

We now describe the adjunction \eqref{qewfoihqoiewfqwfqwefqewf}. We first explain the functor $U$.
Let $\bD$ be in $\nClincat$. Then $U(\bD)$ in $\Clincat$ is defined as follows:
\begin{enumerate}
\item objects: The objects of $U(\bD)$ are pairs $(D,p_{D})$ of an object $D$ of $\bD$ and a selfadjoint projection $p$ in $\End_{\bD}(D)$.
\item morphisms: The $\C$-vector space of morphisms $\Hom_{U(\bD)}((D,p_{D}),(D',p_{D'}))$ is defined as the subspace
$p_{D'}\Hom_{\bD}(D,D')p_{D}$ of $\Hom_{\bD}(D,D')$.
\item composition and  involution: The composition and the involution are inherited from $\bD$.
\end{enumerate}
If $\phi:\bD\to \bD^{\prime}$ is a morphism in $\nClincat$, then we define the morphism
$U(\phi):U(\bD)\to U(\bD')$ in $\Clincat$ as follows:
\begin{enumerate}
\item objects: The functor $U(\phi)$ sends the object $(D,p_{D})$ in $U(\bD)$ to the object $(\phi(D),\phi(p_{D}))$ in $U(\bD^{\prime})$.
\item morphisms: The action of $U(f)$ on morphisms is defined by restriction of the action of $f$.  
\end{enumerate}
Note that $U(\bD)$ is a small unital $\C$-linear $*$-category, and that the functor $U(\phi)$ is unital. Indeed, the identity of the object $(D,p_{D})$ in $U(\bD)$ is $p_{D}$. 
This finishes the description of the functor $U$.

The counit of the adjunction \eqref{qewfoihqoiewfqwfqwefqewf} is given by the family $(\omega_{\bD})_{\bD\in \nClincat}$  of morphisms
$$\omega_{\bD}:\incl(U(\bD))\to \bD\ .$$ Here $\omega_{\bD}$  sends  the object $(D,p_{D})$  of $\incl(U(\bD))$ to the object $D$ of $\bD$ and is given  by the canonical inclusion on the level of morphisms. 

For $\bC$ in $\Clincat$ and $\bD$ in $\nClincat$  we consider the map 
  \begin{equation}\label{weoivoeiwvwervewrvwervrev}
 \Hom_{\Clincat}( \bC,U(\bD)) \to \Hom_{\nClincat}(\incl(\bC),\bD)
\end{equation}
 which 
sends   $\phi:\bC \to  U(\bD)$  to the composition 
 $$\incl(\bC)\stackrel{\incl(\phi)}{\to} \incl(U(\bD))\stackrel{\omega_{\bD}}{\to} \bD\ .$$
 It is straightforward to check that   \eqref{weoivoeiwvwervewrvwervrev}  is a bijection and bi-natural in $\bC$ and $\bD$. 
  This finishes the description of the adjunction  \eqref{qewfoihqoiewfqwfqwefqewf}.
%
%
\end{proof}

We consider possibly non-unital $*$-algebras over $\C$ as possibly non-unital $\C$-linear $*$-categories with a single object. 
In this way we get a fully faithful inclusion
 \begin{equation}\label{}\nAlgc \to \nClincat \end{equation}
 of the category of  possibly non-unital $*$-algebras over $\C$ and algebra homomorphisms into the category of possibly nonunital $\C$-linear $*$-categories.  We then have a pull-back square of categories
  $$\xymatrix{\Algc\ar[r]\ar[d]&\nAlgc\ar[d]\\\Clincat\ar[r]&\nClincat}\ ,$$
 where $\Algc$ is the category of unital $*$-algebras over $\C$ and morphisms.

\begin{rem}
The adjunction  \eqref{egrggqwwefeqwfr4r41} restricts to an adjunction   \begin{equation}\label{egrggqwwefeqwf1}  (-)^{+}:\nCalg\leftrightarrows \Calg:\incl\ .  \end{equation} 
In contrast, the adjunction  \eqref{qewfoihqoiewfqwfqwefqewf} does not have a counterpart in algebras since the functor
$U$ does not preserve categories with a single object. In fact, the inclusion functor $\incl:\Algc \to \nAlgc$ is not a left adjoint functor since  it does not preserve inital objects. The initial object of $\Algc$ is $\C$, while the initial object of $\nAlgc$ is the zero algebra. \hB
\end{rem}
 
 We consider the inclusion $\incl: \nAlgc\to \nClincat$.
 \begin{lem}
 The inclusion functor is the right-adjoint of 
 an adjunction  \begin{equation}\label{erwgwergwgrevgweg234r3f} A^{f,\alg}:\nClincat\leftrightarrows \nAlgc:\incl\ .\end{equation} 
 \end{lem}
 \begin{proof} The functor 
 $A^{f,\alg}$ sends   $\bC$ in $\nClincat$ to the free $*$-algebra over $\C$ generated by the morphisms of $\bC$
  subject to the relations given by the possible compositions in $\bC$, the $*$-operation, and the linear structure of the $\Hom$-vector spaces  \cite[Def. 3.7]{joachimcat}. The unit of the adjunction
\eqref{erwgwergwgrevgweg234r3f} is the family $(\delta^{\alg}_{\bC})_{\bC\in \nClincat}$ of  morphisms \begin{equation}\label{qewfpoopqwefqwefqwf}
\delta^{\alg}_{\bC}:\bC\to \incl(A^{f,\alg}(\bC))\ .
\end{equation}
Here $\delta^{\alg}_{\bC}$
 sends all objects of $\bC$ to the unique object of $\incl(A^{f,\alg}(\bC))$, and a morphism $f$ in $\bC$ to the corresponding generator of $\incl(A^{f,\alg}(\bC))$. For $\bC$ in $\nClincat$ and $B$ in $\nAlgc$ we consider the map
  \begin{equation}\label{wergkwerpowervwer}
\Hom_{\nAlgc}(A^{f,\alg}(\bC),B)\to \Hom_{\nClincat}(\bC,\incl(B))
\end{equation} which
  sends $\phi:A^{f,\alg}(\bC)\to B$ to the composition 
$$\bC\stackrel{\delta^{\alg}_{\bC}}{\to} \incl(A^{f,\alg}(\bC))\stackrel{\incl(\phi)}{\to} \incl(B)\ .$$
 It is straightforward to check that   \eqref{wergkwerpowervwer}  is a bijection and bi-natural in $\bC$ and $B$. 
\end{proof}

 We have a functor  $$\Ob:\nClincat\to \Set\ , \quad \bC\mapsto \Ob(\bC)$$
sending   an object of $\nClincat$  to its set of objects.

\begin{lem}\mbox{} \begin{enumerate} \item The functor $\Ob$ is the left-adjoint of  an adjunction  
  \begin{equation}\label{qfewfq}\Ob:\nClincat\leftrightarrows \Set: 0[-]\ . \end{equation}
  \item  The functor $\Ob$ is the  right-adjoint of  an adjunction  
  \begin{equation}\label{qfewfewfwefwefewfefwefwffq}0[-]: \Set \leftrightarrows  \nClincat : \Ob \ . \end{equation}
\item The restriction of $\Ob$ to $\Clincat$ is the left-adjoint of an adjunction \begin{equation}\label{qfewfq1}\Ob:\Clincat\leftrightarrows \Set: 0[-] \end{equation} obtained by restriction of    \eqref{qfewfq}. 
  \item  The restriction of $\Ob$ to $\Clincat$
  is the right-adjoint of an adjunction \begin{equation}\label{asdvasvdsvasvadvsv}\C[-]:\Set\leftrightarrows \Clincat:\Ob\ . \end{equation} 
  \end{enumerate}
\end{lem}
\begin{proof}
We describe the adjunction \eqref{qfewfq}. The functor $0[-]$ sends a set $X$ to the category $0[X]$ whose set of objects is $X$, and where all objects are zero objects \cite[Example 2.5]{startcats}. The action of $0[-]$ on maps between sets is clear.  For $\bC$ in $\nClincat$ and $X$ in $\Set$ we consider  the map
 $$ \Hom_{\nClincat}(\bC,0[X])\to \Hom_{\Set}(\Ob(\bC),X)$$    which sends a functor $\bC\to 0[X]$ to its action on the sets of objects. It is straightforward to check that the map is a bijection and bi-natural in $\bC$ and $X$.
 
The adjunction \eqref{qfewfewfwefwefewfefwefwffq} is provided by the bi-natural isomorphism 
$$\Hom_{\Clincat}(0[X],\bC)\to \Hom_{\Set}(X,\Ob(\bC))$$  for $\bC$ in $\nClincat$ and $X$ in $\Set$ 
which sends  $\phi:0[X]\to \bC$ to is action on the set of objects. 

In order to get the adjunction \eqref{qfewfq1} we just observe that $0[-]$ factorizes over $\Clincat$.

The left adjoint $\C[-]$ of the adjunction \eqref{asdvasvdsvasvadvsv}
is given as the composition  $\C[-]:=0[-]^{+}$ of the functor $0[-]$ and the unitalization $(-)^{+}$.
The bi-natural  isomorphism
$$ \Hom_{\Clincat}(\C[X],\bC)\to \Hom_{\Set}(X,\Ob(\bC))$$
for $\bC$ in $\Clincat$ and $X$ in $\Set$ sends a functor $\C[X]\to \bC$ to its action on the sets of objects.
 \end{proof}

\section{$C^{*}$-categories}\label{eqwrughuqiorgfwefwqefqwefqewf}
 
Usually a $C^{*}$-category is defined as a $\C$-linear $*$-category with the additional structure of norms on the   morphism   vector   spaces    \cite{ghr}, \cite{mitchc}, \cite[Def. 2.1]{DellAmbrogio:2010aa}.  One requires, that  the  norms behave  sub-multiplicative with respect to the  composition, that the morphism spaces are complete, and that  a version of the    $C^{*}$-condition \cite[Def. 2.1 (iv)']{DellAmbrogio:2010aa}  is satisfied.  A functor between $C^{*}$-categories is a functor between $\C$-linear $*$-categories which is addition norm-continuous on the morphism spaces.

   But it turns out that being a $C^{*}$-category is actually a property of a  $\C$-linear $*$-category.  
   Moreover, a morphism of     $\C$-linear $*$-categories between $C^{*}$-categories is automatically continuous, i.e., a morphism between $C^{*}$-categories. We can thus consider the category of small $C^{*}$-categories as a  
  full subcategory of the category $\C$-linear $*$-categories. There are unital and non-unital variants.
 In the following we introduce unital and non-unital $C^{*}$-categories from this point of view.   
   
Recall that a $C^{*}$-algebra $B$ is  an object of $\nAlgc$ whose underlying complex vector space  admits a    norm $\|-\|_{B}$  with the following properties:
\begin{enumerate}
\item    $B$ is  complete with respect to the metric induced by $\|-\|_{B}$.
 \item   For all $b,b'$ in $B$ we have $\|bb'\|_{B}\le \|b\|_{B}\|b'\|_{B}$.
\item  For all $b$ in $B$ we have $\|b^{*}b\|_{B}=\|b\|_{B}^{2}$.
\end{enumerate}
Note that $\|-\|_{B}$ is uniquely determined by the $*$-algebra $B$ so that the notation $\|-\|_{B}$ is unambiguous.
  
 We consider an object $\bC$  in $\nClincat$ and a morphism $f$ in $C$.
 \begin{ddd}\label{egergwefqwefqwfqewf}
 We define the maximal semi-norm $ 
\|f\|_{\max}$   of $f$  as the element  \begin{equation}\label{wergwergewgwergwerg}\|f\|_{\max}:=\sup_{\rho}\|\rho(f)\|_{B}\ ,\end{equation} of $[-\infty,\infty]$,
 where the supremum runs over all morphisms $\rho:\bC\to B$ in $\nClincat$ with target a    $C^{*}$-algebra $B$ . 
 \end{ddd}
 Note that we always have a morphism $\bC\to 0[*]$. Hence the index set of the supremum in \eqref{wergwergewgwergwerg} is always non-empty and therefore $\|-\|_{\max}$ takes values in $[0,\infty]$.

The following Definition \footnote{Warning: The notion of a pre-$C^{*}$-category according to Definition \ref{ruighqregqrgqfewfqqf} differs from the notion   defined in \cite[Def. 2.1]{DellAmbrogio:2010aa}.}
  is a straightforward generalization of \cite[Def. 2.10]{startcats} to the non-unital case. Let $\bC$ be in $\nClincat$.
\begin{ddd}\label{ruighqregqrgqfewfqqf} $\bC$ is a pre-$C^{*}$-category if all morphisms in $\bC$  have a finite maximal semi-norm. \end{ddd}
We let $\npClincat$ denote the full subcategory of $\nClincat$ of pre-$C^{*}$-categories.
\begin{lem}
The inclusion is the left-adjoint of an adjunction \begin{equation}\label{wefqwefefefqwfqfeewfewfqewf}
\incl:\npClincat\leftrightarrows \nClincat:\Bd^{\infty}\ .
\end{equation}
\end{lem}
 \begin{proof}  This lemma is 
 the straightforward generalization of 
\cite[Lemma 3.8]{startcats} to the non-unital case\footnote{Thereby we take the chance to correct a mistake in \cite[Lemma 3.8]{startcats}.
In the reference we defined the functor $\Bd^{\infty}$ as a countable iteration of the functor $\Bd$ in order to ensure the relation
$\Bd(\Bd^{\infty}(\bC))\cong \Bd^{\infty}(\bC)$. But in general  this formula is only correct if we define $\Bd^{\infty}$ as a sufficiently large transfinite iteration of $\Bd$ as is done in the present paper.}.
In order to describe the functor $\Bd^{\infty}$, as a first approximation we
consider the endo-functor
$$\Bd:\nClincat\to \nClincat$$ defined as follows. Let $\bC$ be in $  \nClincat$.
Then $\Bd(\bC)$ has the following description:
\begin{enumerate}
\item objects: The set of objects of $\Bd(\bC)$ is the set of objects of $\bC$.
\item For objects $C,C^{\prime}$ in $\bC$ we have
$\Hom_{\Bd(\bC)}(C,C^{\prime}):=\{f\in \Hom_{\bC}(C,C^{\prime})\:|\: \|f\|_{\max}<\infty\}$.
\end{enumerate}
One checks that $\Bd(\bC)$ is a wide $\C$-linear $*$-subcategory of $\bC$.
In order to define $\Bd$ on morphisms we observe that if $\phi:\bC\to \bC'$ is a morphism in $\nClincat$, then
$\phi$ sends $\Bd(\bC)$ to $\Bd(\bC^{\prime})$. We define $\Bd(\phi)$ as the restriction of $\phi$ to $\Bd(\bC)$.

We have a canonical inclusion $\kappa_{\bC}:\Bd(\bC)\to \bC$.

By transfinite induction we now construct a family, indexed by ordinals $\alpha$, of functors  $ \Bd^{\alpha} :\nClincat\to \nClincat$  together with  transformations
$\kappa^{\alpha} : \Bd^{\alpha} \to \id$  which on each object are  inclusions of subcategories.
 \begin{enumerate}
\item $\Bd^{0}:=\id$.
\item If $\alpha$ is a successor ordinal, i.e., $\alpha= \beta+1$, then we set  $\Bd^{\alpha}:=\Bd\circ \Bd^{\beta}$, and $\kappa^{\alpha}:=   \kappa \circ \Bd(\kappa^{\beta}) $.
\item If $\alpha$ is a limit ordinal, then we define
$\Bd^{\alpha}:=\lim_{\beta<\alpha} \Bd^{\beta}$ and  let
$\kappa^{\alpha}$ be the evaluation of the limit at $\beta=0$.
\end{enumerate}
Note  $(\Bd^{\alpha}(\bC))_{\alpha}$ is a decreasing family of wide subcategories of $\bC$.

We now define a functor
$$\Bd^{\infty}:\nClincat\to \npClincat$$ as follows:
\begin{enumerate}
\item objects:
Given an object  $\bC$ in $\nClincat$  there exists an ordinal $\alpha$ (depending on $\bC$) such that the canonical morphism $\Bd^{\alpha'}(\bC)\to  \Bd^{\alpha}(\bC)$ is an isomorphism for all $\alpha'\ge \alpha$. 
It suffices to take $\alpha$ larger then the size of the union of the morphism spaces of $\bC$.
This implies that   $$\Bd^{\infty}(\bC):=\lim_{\alpha}\Bd^{\alpha}(\bC)$$  (the limit is an intersection) exists and is a pre-$C^{*}$-category.  
\item morphisms:
 If $\phi:\bC\to \bC^{\prime}$ is a morphism in $\nClincat$, then
we  define
$$\Bd^{\infty}(\phi):\Bd^{\infty}(\bC)\to \Bd^{\infty}(\bC')$$
as $\Bd^{\alpha}(\phi)$ for sufficiently large $\alpha$.
 \end{enumerate} 
The functor $\Bd^{\infty}$ comes with a natural transformation 
   $$\kappa^{\infty}:\incl(\Bd^{\infty})\to \id$$
   which is the counit
  of the adjunction  \eqref{wefqwefefefqwfqfeewfewfqewf}.
  In general, for $\bC$ in $\nClincat$ the functor $\kappa_{\bC}^{\infty}$ is the inclusion of a wide subcategory, but if 
 $\bC$ is  a   pre-$C^{*}$-category, then $\kappa^{\infty}_{\bC}$ is an isomorphism (actually an equality).
  For $\bD$ in $\npClincat$ and $\bC$ in $\nClincat$ we consider the map \begin{equation}\label{wergopjoipwergergwergwrg} \Hom_{\npClincat}(\bD,\Bd^{\infty}(\bC))
 \to \Hom_{\nClincat}(\incl(\bD),\bC)
\end{equation}
  which sends a morphism 
$\phi:\bD\to \Bd^{\infty}(\bC)$ to the composition
$$\incl(\bD) \stackrel{\incl(\phi)}{\to} \incl(\Bd^{\infty}(\bC))\stackrel{\kappa^{\infty}_{\bC} }{\to} \bC \ .$$
It is straightforward to check that  \eqref{wergopjoipwergergwergwrg} is bi-natural in $\bD$ and $\bC$. 
As in the proof of  \cite[Lemma 3.8]{startcats} one checks that it is a bijection.
 \end{proof}


Let $\bC$ be in $\npClincat$.
\begin{ddd}
$\bC$  is a $C^{*}$-category if its morphism spaces are complete with respect to the maximal norm.
\end{ddd}

\begin{rem} 
Note that the maximal norm is in general a semi-norm, i.e., non-zero elements might have zero maximal semi-norm. 
Completeness in particular involves the condition that the maximal semi-norm is a norm. \hB
\end{rem}

The category of possibly non-unital $C^{*}$-categories and morphisms is the full subcategory 
of $\nClincat$ consisting of $C^{*}$-categories.
 We have a diagram of pull-back squares \begin{equation}\label{qrgqrgqrgqrg}
 \xymatrix{\Ccat\ar[r]\ar[d]&\nCcat\ar[d]\\\pClincat\ar[r]\ar[d]&\npClincat\ar[d]\\\Clincat\ar[r] &\nClincat } 
\end{equation}
defining the categories $\pClincat$ of unital pre-$C^{*}$-categories and unital $C^{*}$-categories $\Ccat$.

\begin{lem}\mbox{}
\begin{enumerate}
\item We have an adjunction  \begin{equation}\label{oidfhbuihsfuibfefefesdfbsfb}\Compl:\npClincat\leftrightarrows \nCcat:\incl \ . \end{equation}
\item The adjunction \eqref{oidfhbuihsfuibfefefesdfbsfb} restricts to an adjunction 
 \begin{equation}\label{oidfhbuihsfuibfefefesdfbsfb1}\Compl:\pClincat\leftrightarrows \Ccat:\incl \ . \end{equation}
\end{enumerate}
\end{lem}
\begin{proof}
We first describe the completion functor $\Compl$. Let $\bC$ be in $\npClincat$. Then
$\Compl(\bC)$ has the following description:
\begin{enumerate}
\item objects: The set of objects of $\Compl(\bC)$ is the set of objects of $\bC$.
\item morphisms:  For objects $C ,C^{\prime}$ in $\bC$ the space of morphisms $\Hom_{\Compl(\bC)}(C,C')$ is obtained from 
$\Hom_{ \bC}(C,C')$ by first forming the quotient by the subspace of vectors of zero maximal seminorm (zero-morphisms), and then forming the metric completion.
\item composition and involution: We observe that the composition of any morphism with a zero morphism  and the adjoint of a zero morphism  are  again   zero morphisms. Hence we get an induced composition or involution on the quotient morphism spaces which then extends by continuity to the completions.  \end{enumerate}
Let $\phi:\bC\to \bC^{\prime}$ be a morphism in $\npClincat$. We observe that
$\phi$ preserves zero-morphisms. Hence it induces maps between the quotients of morphism spaces by zero-morphisms. 
Then the morphism $\Compl(\phi)$ is the defined from these induced maps by continuous extension.

The unit $\alpha: \id\to  \incl\circ \Compl $ of the adjunction \eqref{oidfhbuihsfuibfefefesdfbsfb} is given by
 is the canonical morphisms \begin{equation}\label{qwefkjiqkwefqwefqewf}\alpha_{\bC}:\bC\to \incl(\Compl(\bC)) \end{equation}  for all $\bC$ in $\npClincat$.  For $\bC$ in $\npClincat$ and $\bD$ in $  \nCcat$
we consider the map 
  \begin{equation}\label{wevwervwervfvvsdfv}
 \Hom_{\npClincat}(\Compl(\bC), \bD ) \to \Hom_{\npClincat}(\bC,\incl(\bD))
\end{equation}
 which sends a morphism $\phi:\Compl(\bC)  \to   \bD $ to the composition
$$\bC\stackrel{\alpha_{\bC}}{\to}\incl(\Compl(\bC))\stackrel{\incl( \phi) }{\to}   \incl(\bD)\ .$$
It is straightforward to check that \eqref{wevwervwervfvvsdfv}  is bi-natural in $\bC$ and $\bD$, and easy to see that it is a bijection \cite[Rem. 3.3]{startcats}.

In order to get the adjunction \eqref{oidfhbuihsfuibfefefesdfbsfb1}  from \eqref{oidfhbuihsfuibfefefesdfbsfb} we just observe that the completion of a unital pre-$C^{*}$-category is a unital $C^{*}$-category.
\end{proof}
 
 The unitalization of $C^{*}$-categories has been considered already in  \cite[Prop. 3.4 \& 3.5]{mitchc}. \begin{lem}
 We have adjunctions \begin{equation}\label{egrggqwwefeqwfr4r411}  (-)^{+}:\nCcat\leftrightarrows \Ccat:\incl\ ,  \end{equation} 
and
 \begin{equation}\label{qewfoihqoiewfqwfqwefqewf1}
\incl:\Ccat\leftrightarrows \nCcat:U\ .
\end{equation}\end{lem}
\begin{proof}
These adjunctions are obtained by restricting the adjunctions \eqref{egrggqwwefeqwfr4r41}
and \eqref{qewfoihqoiewfqwfqwefqewf} to $C^{*}$-categories.
We just observe that the functors $(-)^{+}$ and $U$ preserve $C^{*}$-categories.
\end{proof}

\begin{lem}\mbox{} \begin{enumerate} \item The functor $\Ob$ is the left-adjoint of an adjunction 
  \begin{equation}\label{1qfewfq}\Ob:\nCcat\leftrightarrows \Set: 0[-]\ . \end{equation}   \item The functor $\Ob$ is the  right-adjoint of  an adjunction 
    \begin{equation}\label{1qfevdfvdfvfdvfdwfq}0[-]  :\Set  \leftrightarrows \nCcat : \Ob\ . \end{equation}
  \item The restriction of $\Ob$ to $\Ccat$ is the left adjoint of  the adjunction \begin{equation}\label{1qfewfq1}\Ob:\Ccat\leftrightarrows \Set: 0[-] \end{equation}
  obtained by restriction of \eqref{1qfewfq}.
  \item  The restriction of $\Ob$ to $\Ccat$
  is the right-adjoint of an adjunction \begin{equation}\label{1asdvasvdsvasvadvsv}\C[-]:\Set\leftrightarrows \Ccat:\Ob\ . \end{equation} 
  \end{enumerate}
\end{lem}
\begin{proof}
The adjunctions are obtained by restricting the adjunctions \eqref{qfewfq},  \eqref{qfewfewfwefwefewfefwefwffq}, \eqref{qfewfq1} and \eqref{asdvasvdsvasvadvsv}. To this end we observe that $\C[-]$ and $0[-]$ take values in $C^{*}$-categories.
\end{proof}

The adjunction  \eqref{erwgwergwgrevgweg234r3f} has a counterpart in the   $C^{*}$-case. The following is   \cite[Def. 3.7]{joachimcat}.
\begin{lem}
We have 
 an adjunction  \begin{equation}\label{evwervwervwevwev} A^{f }:\nCcat\leftrightarrows \nCalg:\incl\ .\end{equation}
 \end{lem}
 \begin{proof} We define the category of pre-$C^{*}$-algebras as the intersection \begin{equation}\label{wqrefonwiefjqwefewqfw}
 \npAlgc:=\nAlgc\cap \npClincat
\end{equation}
  in $\nClincat$. The adjunction \eqref{oidfhbuihsfuibfefefesdfbsfb} restricts to an adjunction  \begin{equation}\label{ervevwevefvdsfvsdfvsdfv}
\Compl: \npAlgc\leftrightarrows \nCalg:\incl\ .
\end{equation}

   The functor $A^{f}$  is given  by the composition  \begin{equation}\label{ewvqwvwccxewcqwcxwd} A^{f}:\nCcat\stackrel{A^{f,\alg},\eqref{erwgwergwgrevgweg234r3f}}{\to} \npAlgc \stackrel{\Compl, \eqref{ervevwevefvdsfvsdfvsdfv}}{\to} \nCalg\ .  \end{equation}    
   We must check that     the restriction of $A^{f,\alg}$ to $C^{*}$-categories takes values in pre-$C^{*}$-algebras. To this end we note that for a $C^{*}$-category $\bC$ and a
 morphism $\rho:A^{f,\alg}(\bC)\to B$ into a $C^{*}$-algebra $B$  by precomposing it with the unit of the adjunction    \eqref{erwgwergwgrevgweg234r3f} we get a morphism 
 $$\tilde \rho:\bC\stackrel{\delta^{\alg}_{\bC}, \eqref{qewfpoopqwefqwefqwf} }{\to} A^{f,\alg}(\bC)\to B\ .$$ For every morphism $f$ in $\bC$ we   have the inequality  $$\|\rho(\delta^{\alg}_{\bC}(f))\|_{B}=\|\tilde \rho(f)\|_{B}\le \|f\|_{\bC}\ .$$ Varying $\rho$ and $B$
 we conclude that 
  for every morphism $f$ in $\bC$ we have
 $$\|\delta^{\alg}_{\bC}(f)\|_{\max}\le \|f\|_{\bC}\ .$$ Since every element of $A^{f,\alg}(\bC)$ is a finite linear combination of finite products of morphisms of the form $\delta^{\alg}_{\bC}(f)$ we conclude that every element of $ A^{f,\alg}(\bC)$ has finite maximal norm.  
 The unit of the adjunction
  \eqref{evwervwervwevwev} is the
 natural transformation $\delta:\id\to \incl\circ A^{f}$ given by 
 $$\delta_{\bC}:\bC\stackrel{\delta^{\alg}_{\bC}, \eqref{qewfpoopqwefqwefqwf} }{\to} \incl(A^{f,\alg}(\bC))\stackrel{\alpha_{ \incl(A^{f,\alg}(\bC))},\eqref{qwefkjiqkwefqwefqewf}}{\to} \incl(\Compl( \incl(A^{f,\alg}(\bC))))=\incl(A^{f}(\bC))$$
 for every $\bC$ in $\nCcat$.
 For $\bC$ in $\nCcat$ and $B$ in $\nCalg$ we define the map
 \begin{equation}\label{wegojgkpwoergwergwergwerg}
\Hom_{\nCalg}(A^{f}(\bC),B)\to \Hom_{\nCcat}(\bC,\incl(B))
\end{equation} 
 which sends 
 $\phi:A^{f}(\bC)\to B$ to
 $$\bC\stackrel{\delta_{\bC}}{\to}\incl(A^{f}(\bC)) \stackrel{\incl(\phi)}{\to} \incl(B)\ .$$
 It is straightforward to check that \eqref{wegojgkpwoergwergwergwerg} is bi-natural in $\bC$ and $B$ and an isomorphism.
   \end{proof}

\section{Completeness and cocompleteness of $\nClincat$, $\npClincat$ and $\nCcat$}


A category is called complete  if it admits limits  for all diagrams indexed by small categories.  Similarly, a category is called  cocomplete,  if it admits colimits   for all diagrams indexed by small categories. It is known that  the categories
$\Clincat$, $\pClincat$ and $\Ccat$ are complete and cocomplete, see   \cite{DellAmbrogio:2010aa} (for $\Ccat$) or \cite[Thm. 8.1]{startcats} for arguments.
 In this section we show that this result   extends to the non-unital case.

\begin{theorem}\label{riguhqwieufqewfeqfqewf}
\mbox{}
The categories $\nClincat$, $\npClincat$ and $\nCcat$  are complete and cocomplete. 
\end{theorem}

The main idea of the proof of this theorem is to reduce the assertion to the corresponding assertion in the unital case. This reduction is based on the following constructions. As usual we let  $\Delta^{1}$ denote  the category of the shape $\bullet \to \bullet$. Then for a category $\cC$ the functor category  $\cC^{\Delta^{1}}$ is the category of morphisms in  $\cC$.
 \begin{ddd}\label{wtiwegerwergergwreg}For a category $\cC$ with an endofunctor $F:\cC\to \cC$ 
  we let $\cC_{F}$ denote the full subcategory of $\cC^{\Delta^{1}}$ on  objects of the form $\bC\to F(\bC)$ for objects $\bC$ of $\cC$. \end{ddd}

Let $F:\cC\to \cC$ be an endofunctor,  and let $\bI$ be a small category.  \begin{lem}\label{efbwebwervervwev} \mbox{}
\begin{enumerate} \item  \label{fvkfovavsdvasvadsvasdv1} If $\cC$  admits $\bI$-shaped colimits     and $F$ preserves  $\bI$-shaped colimits, then $\cC_{F}$ admits  $\bI$-shaped colimits. 
\item \label{fvkfovavsdvasvadsvasdv2} If $\cC$ admits  $\bI$-shaped  limits  and $F$ preserves   $\bI$-shaped limits, then $\cC_{F}$  admits  $\bI$-shaped  limits. \end{enumerate}
\end{lem}\begin{proof}
 If $\cC$  admits $\bI$-shaped colimits, then    the functor category $\cC^{\Delta^{1}}$    admits $\bI$-shaped colimits. Furthermore, if $F$ preserves  $\bI$-shaped  colimits, then 
 the full subcategory 
$\cC_{F}$ of $\cC^{\Delta^{1}}$ is   closed under  $\bI$-shaped colimits   and hence itself admits $\bI$-shaped colimits. 

The argument for limits is similar.
\end{proof}

 We apply this construction and lemma to the categories  $\Clincat$ and  $\Ccat$  in place of $\cC$ and the 
 endofunctor \begin{equation}\label{favsvdsdav}
F:=\C[\Ob(-)]\ .
\end{equation}

Let $p:\bC\to \bD$ be a morphism in $\nClincat$. Then we can form the wide  subcategory $\Ker(p)$ of $\bC$ as an object  in $\nClincat$ as follows:
\begin{enumerate}
\item objects: The set of objects of $\Ker(p)$ is the set of objects of $\bC$.
\item morphisms: The $\C$-vector space of morphisms between objects $C,C^{\prime}$ of $\bC$ is given by
$$\Hom_{\Ker(p)}(C,C^{\prime}):=\ker\left(\Hom_{\bC}(C,C')\to \Hom_{\bD}(p(C),p(C'))\right)\ .$$
\item composition and involution: These structures are inherited from $\bC$.
\end{enumerate}

 We define a functor
$$\beta:(\Clincat)_F\to  \nClincat$$ as follows:
\begin{enumerate}
\item objects: The functor $\beta$ sends the object $p:\bC\to F(\bC)$ in $(\Clincat)_F$  to $\Ker(p)$ in $\nClincat$.
\item morphisms: Let $\phi:(p:\bC\to F(\bC))\to (p':\bC'\to F(\bC^{\prime}))$  be a morphism in $(\Clincat)_F$,  i.e. a commutative square
$$\xymatrix{\bC\ar[r]\ar[d]&\bC^{\prime}\ar[d]\\F(\bC)\ar[r]&F(\bC^{\prime})}\ .$$ 
Then functor  $\bC\to \bC^{\prime}$  restricts to a functor 
 $\beta(\phi):\Ker(p)\to \Ker(p^{\prime})$. 
\end{enumerate}

 Since the kernel of a morphism between $C^{*}$-categories is a $C^{*}$-category the functor $\beta$ restricts to a functor
$$\beta :(\Ccat)_F\to  \nCcat \ .$$
 
\begin{lem}\label{ewjiowebwerwerbwerbeb}\mbox{}
The functor $\beta :(\Clincat)_F\to  \nClincat$ is an equivalence. 
\end{lem}
\begin{proof} Let $\incl:\Clincat\to \nClincat$ be the inclusion.
We have  a natural transformation of functors 
$$\id\to  \incl(\C[\Ob(-)]) :\nClincat \to \nClincat $$ which sends 
  $\bC$   in $\nClincat$ to the     morphism $\bC\to \incl(\C[\Ob(\bC)])$   which is the identity on objects and sends all morphisms to zero. 
  
  Taking objectwise the  adjoints with respect to the adjunction \eqref{egrggqwwefeqwfr4r41} we obtain   the natural transformation  of functors  $$\epsilon: (-)^{+}\to \C[\Ob(-)]:\nClincat\to \Clincat\ .$$ 
 The inverse $$(-)^{\dagger}:  \nClincat\to  (\Clincat)_F$$  of $\beta $ is the natural transformation $\epsilon$ interpreted as a functor $ \nClincat\to \Clincat^{\Delta^{1}}$ which happens to take values in the subcategory   $(\Clincat)_F$.
 It sends $\bC$ in $\nClincat$ to $$\bC^{\dagger}:=(\epsilon_{\bC}:\bC^{+}\to \C[\Ob(\bC)])\ .$$

We have an obvious natural isomorphism of functors $\id\cong \beta  ((-)^{\dagger})$. The isomorphism 
$    (\beta(-))^{\dagger}  \stackrel{\cong}{\to} \id $ is given on the object $p:\bC\to \C[\Ob(\bC)]$  of $(\Clincat)_{F}$ by the commuting diagram
 $$\xymatrix{\Ker(p)^{+} \ar[r]^{!}\ar[d]& \bC \ar[d]\\ \C[\Ob(\bC)] \ar[r]^{=}& \C[\Ob(\bC)]}\ ,$$
where the  arrow marked by $!$ is induced by the embedding $\Ker(p)\to \bC$ and the universal property of the unitalization.
It is an isomorphism. This finishes the proof of 
Lemma \ref{ewjiowebwerwerbwerbeb}.
%
%
\end{proof}

\begin{proof}[Proof of Theorem \ref{riguhqwieufqewfeqfqewf}]
In a first step we show the assertion for $\nClincat$.
We first discuss colimits. We already know that 
 that $\Clincat$ is 
 cocomplete.  
The composition $F:= \C[-]\circ \Ob:\Clincat\to\Clincat$
is the composition of two left-adjoints \eqref{asdvasvdsvasvadvsv} and \eqref{qfewfq1}  and therefore preserves all small colimits.  
By Lemma \ref{efbwebwervervwev}.\ref{fvkfovavsdvasvadsvasdv1} it follows that
$(\Clincat)_{F}$ is 
cocomplete. 
Finally, by Lemma \ref{ewjiowebwerwerbwerbeb} the category 
$\nClincat$ is 
cocomplete. 
%
%
%

We now consider limits. 
The argument for colimits does note completely apply  to limits since the functor $\C[\Ob(-)]$ does only preserve limits of connected shape. It does not preserve   products in general.
Nevertheless, completeness of $\nClincat$  follows from the following assertions:
\begin{enumerate}
\item\label{rguiqegergwergergregw} existence of final objects,
\item\label{rguiqegergwergergregw1} existence of limits with connected shape,
\item\label{rguiqegergwergergregw2} existence of products.
\end{enumerate}

We start with Assertion \ref{rguiqegergwergergregw}.
The category $0[*]$ is a final object of $\nClincat$. 

We now consider Assertion \ref{rguiqegergwergergregw1}.
 A small category $\bI$ is called connected if 
its nerve is a connected simplicial set. Equivalently, $\bI$ is connected  iff every two objects in $\bI$ are connected by a composition of  zig-zags.

Assume now that $\bI$ is a non-empty connected small category. We claim that the functor 
$\C[-]:\Set\to \Clincat$ preserves $\bI$-shaped limits.

In order to see this claim we consider $X$ in $\Fun(\bI,\Set)$. We must show that  the canonical morphism $$\C[\lim_{\bI}X]\to \lim_{\bI}\C[X]$$ is an isomorphism. To this end we show that the morphism  \begin{equation}\label{dsvasdvqewfsdvad}
\underline{\C[\lim_{\bI}X]}\to \C[X]
\end{equation} 
induced by the canonical morphism $\underline{\lim_{\bI}X}\to X$
presents 
$\C[\lim_{\bI}X]$ as the limit of the diagram $\C[X]$. 
Hence   we must show that the post-composition with  the morphism in \eqref{dsvasdvqewfsdvad}  induces a bijection \begin{equation}\label{ewrgknhkwergergewrg}
\Hom_{\Clincat}(\bT,\C[\lim_{\bI}X])\to \Hom_{\Fun(\bI,\Clincat)}(\underline{\bT},\C[X])
\end{equation}
for every $\bT$ in $\Clincat$.  
In order to describe the inverse of \eqref{ewrgknhkwergergewrg} we view $\lim_{\bI}X$ as a subset of $\prod_{i\in I} X(i)$ in the canonical way.

Let $\phi:\underline{\bT}\to \C[X]$ be given. Note that $\phi$ is given by a compatible collection of functors $\phi(i):\bT\to \C[X(i)]$ for all $i$ in $\bI$.
If $f:T\to T'$ is a morphism in $\bT$ and  $i$ is in $\bI$ such that   $\phi(i)(T)=\phi(i)(T')$, then we get a number  $c_{\phi}(i,f)$ in $\C$ characterized by $$ c_{\phi}(i,f)\id_{\phi(i)(T)}=\phi(i)(f)\ .$$

The  inverse of \eqref{ewrgknhkwergergewrg} sends
$\phi:\underline{\bT}\to \C[X]$ to the functor
$\bT\to \C[\lim_{\bI}X]$ which has the following description:
\begin{enumerate}
\item objects: It sends the object $T$ of $\bT$ to the family $(\phi(i)(T))_{i\in \bI}$ in $\lim_{\bI}X$.
\item morphisms: It sends a morphism $f:T\to T'$ in $\bT$ to the morphism $(\phi(i)(T))_{i\in \bI}\to (\phi(i)(T'))_{i\in \bI}$ given by 
\begin{enumerate}
\item $0$ if $ (\phi(i)(T))_{i\in \bI}\not= (\phi(i)(T'))_{i\in \bI}$  \item $c_{\phi}(i,f)$  for some choice of $i$ in $\bI$ if $ (\phi(i)(T))_{i\in \bI}= (\phi(i)(T'))_{i\in \bI}$. Since we assume that $\bI$ is non-empty and connected 
 the number $ c_{\phi}(i,f)$ is defined and  does not depend on the choice of $i$. 
  \end{enumerate}
 \end{enumerate}
One easily checks that this describes an inverse to \eqref{ewrgknhkwergergewrg}.
 
Note that $\C[-]:\Set \to \Ccat$ preserves $\bI$-shaped limits  by the same argument.

  Since $\Ob$ is a right-adjoint in \eqref{asdvasvdsvasvadvsv} it preserves all limits. Hence the composition    $\C[\Ob(-)]$ preserves $\bI$-shaped limits.
  
 Since $\Clincat$ is 
  complete   we can use
Lemma \ref{efbwebwervervwev}.\ref{fvkfovavsdvasvadsvasdv2} to see that
 $(\Clincat)_{F}$ 
  admits  $\bI$-shaped limits.    Finally,  by Lemma \ref{ewjiowebwerwerbwerbeb} the category 
$\nClincat$ 
admits
  $\bI$-shaped limits. 

We finally show  Assertion \ref{rguiqegergwergergregw2}.
 Let $I$ be a set and $(\bC_{i})_{i\in I}$ be a family in $\nClincat$.
 Then we define $\bC$ in $\nClincat$ as follows:
 \begin{enumerate}
 \item objects: The set of objects of $\bC$ is the set $\prod_{i\in I} \Ob(\bC_{i})$.
 \item morphisms: The $\C$-vector space of morphisms between objects $(C_{i})_{i\in I}$ and $(C'_{i})_{i\in I}$ of $\bC$ is defined by
 $$\Hom_{\bC}((C_{i})_{i\in I},(C'_{i})_{i\in I}):=\prod_{i\in I}\Hom_{\bC_{i}}(C_{i},C_{i}')\ .$$
 \item composition and involution:
The composition and involution are  given by the corresponding componentwise operations.
 \end{enumerate}
For every $i$ in $I$ we have an obvious projection $p_{i}:\bC\to \bC_{i}$.
It is easy to check that $(\bC,(p_{i})_{i\in I})$ presents $\bC$ as the product of the family $(\bC_{i})_{i\in I}$ in $\nClincat$.

%

This finishes the proof of the theorem in the case of $\nClincat$. The remaining cases can be deduced in a completely formal way using the following general fact from category theory.

\begin{rem}\label{fefvwvdfvsdv}
 Assume that $$L:\cC\leftrightarrows \cD:R$$
is a reflective localization, i.e., an adjunction such that $R$ is fully faithful.
 \begin{prop}\label{roijgeqroifqewqfewf}\mbox{}\begin{enumerate}
 \item \label{feuegoergwegwegwerg}  If $\cC$ is complete, then so is $\cD$. The functor $R$ preserves and detects limits.
 \item \label{ergergerwf} If $\cC$ is cocomplete, then so is $\cD$. If $D:\bI\to \cD$ is a diagram in $\cD$, then
 $$\colim_{\bI} D \cong L (\colim_{\bI} R(D))\ .$$
 \end{enumerate}
 \end{prop}\begin{proof}\footnote{We think that the proposition is well-known in category theory, but we add the proof as a service for readers in other fields. The authors thanks G. Raptis for valuable hints.}
 Since $R$ is fully faithful, we can identify $\cD$ with the essential image of $R$. We will omit the inclusion from the notation.
 
 We first consider limits.  Let $W$ be the class of morphisms in $\cC$ which are send to isomorphisms by $L$.
 An object $C$ of $\cC$ is called $W$-local if for every $w$ in $W$ the morphism $\Hom_{\cC}(w,C)$ is an isomorphism.
 We claim that $\cD$ consists exactly of the $W$-local objects. Indeed, if $D$ is in $\cD$, then it is $W$-local since
 $\Hom_{\cC}(w,D)\cong  \Hom_{\cC}(L(w),D)$. 
 
 Assume now that $C$ is $W$-local. 
Let $\eta:C\to L(C)$ be the unit of the adjunction. We show that $\eta$ is an isomorphism. This implies that $D\in \cD$.
The map $\eta$ itself belongs to $W$ since $$L(C)\stackrel{L(\eta)}{\to} L(\incl(L(C)))   \stackrel{counit\circ L}{\to}  L(C)$$ 
is an isomorphism by the triple identity of the adjunction (here it is useful to write the inclusion), and the counit is an isomorphism since $R$ is fully faithful. Since we assume that $C$ is $W$-local 
$$\Hom_{\cC}(L(C),C)\stackrel{\eta^{*}}{\to} \Hom_{\cC}(C,C)$$ is an isomorphism.
We let $\kappa:L(C)\to C $ be the preimage of $\id_{C}$. Then by definition $\kappa\circ \eta=\id_{C}$.
This implies that $\kappa\in W$ since $L(\kappa)\circ L(\eta)=\id_{L(C)}$ and $L(\eta)$ is an isomorphism.
Furthermore 
$$\Hom_{\cC}(C,L(C))\stackrel{\kappa^{*}}{\to} \Hom_{\cC}(L(C), L(C))\  .$$
is an isomorphism. Hence
there exists $\delta:C\to L(C)$ such that $\delta\circ \kappa=\id_{L(C)}$. Both equalities togther imply
that $\delta=\eta$ and hence $\eta$ is invertible.

We now show that $\cD$ is closed under limits. Let $D:\bI\to \cD$ be a diagram.
Then for every $w$ in $W$ we have $\Hom_{\cC}(w,\lim_{\bI}D)\cong \lim_{\bI}\Hom_{\cC}(w,D)$.
Since a limit of isomorphisms is an isomorphism we conclude that $ \Hom_{\cC}(w,\lim_{\bI}D)$ is an isomorphism.
Since $w$ is arbitrary we conclude that $  \lim_{\bI}D$ is $W$-local and hence in $\cD$.

Since $\cD\to \cC$ is fully faithful, we can conclude that $\cD$ has all limits. They  are calculated in $\cC$.
 This finishes the proof of Assertion \ref{feuegoergwegwegwerg}.
 
Let $D:\bI\to \cD$ be a diagram in $\cD$.  
Since $\bC$ is cocomplete we can form the colimit $\colim_{\bI} R(D)$.
Its structure maps $(\iota_{i}:R(D_{i})\to\colim_{\bI} R(D))$
induce a bijection
$$\Hom_{\cC}(\colim_{\bI} R(D),R(D'))\cong \lim_{\bI^{\op}} \Hom_{\cD}( R(D),R(D'))\ .$$
 Using the adjunction and the fact that $R$ is fully faithful (and therefore that the counit  $L\circ R \stackrel{\cong}{\to} \id$ is an isomorphism)  we conclude that the structure maps 
 \begin{equation}\label{eqwrqwerewr}
  D_{i}\cong L(R(D_{i}))\stackrel{L(\iota_{i})}{\to} L(\colim_{\bI} R(D)) 
\end{equation}  induce a bijection 
 $$\Hom_{\cD}(L(\colim_{\bI} R(D)),D')\cong \lim_{\bI^{\op}} \Hom_{\cD}(  D,D')\ .$$
Hence the family of structure maps \eqref{eqwrqwerewr} for $i$ in $\bI$
presents $L(\colim_{\bI} R(D))$ as the colimit of the diagram $D$ in $\cD$.
 
 We can conclude that $\cD$ is cocomplete.
 \end{proof}

 By going over to opposite categories we obtain 
 an analogous statement of Proposition \ref{roijgeqroifqewqfewf} for colocalizations. \end{rem}

 Using the colocalization \eqref{wefqwefefefqwfqfeewfewfqewf} we can conclude from the opposite of Proposition \ref{roijgeqroifqewqfewf} and the case of $\nClincat$ already shown above that $\npClincat$ is complete and cocomplete.
 
We then use the localization  \eqref{oidfhbuihsfuibfefefesdfbsfb} and  Proposition \ref{roijgeqroifqewqfewf}  
in order to deduce the case of $\nCcat$ from the assertion for  $\npClincat$.
\end{proof}


In the following
  we provide  an explicit formula for  a filtered colimit in $\nCcat$  of a diagram  of subcategories of a  fixed  $C^{*}$-category.
 We consider a filtered poset $I$ and a diagram $\bC:I\to \nCcat$.
 We let furthermore $\bD$ be in $\nCcat$ and $\phi:\bC\to \underline{\bD}$ be a morphism in $\Fun(I,\nCcat)$,  {where $\underline{\bD}$ is the constant diagram.} We assume that for every $i$ in $I$ the map $\phi(i):\bC(i)\to \bD$ is injective on objects and morphisms. 
We can then define a $\C$-linear $*$-subcategory $\bE$ of $\bD$ as follows:
\begin{enumerate} 
\item objects: The set of objects of $\bE$ is given as a subset of $\Ob(\bD)$ by
$$\Ob(\bE):= \bigcup_{i\in I} \phi(i)(\Ob(\bC(i)))\ .$$
\item  morphisms:  The   morphism spaces  of $\bE$ are given as subspaces of the morphism spaces of $\bD$ by
$$\Hom_{\bE}(D,D')= \bigcup_{i\in I_{D,D'}} \phi(i)(\Hom_{\bC(i)}(D_{i},D_{i}'))\ ,$$
where $$I_{D,D'}:=\{i\in I\:|\: (\exists D_{i},D_{i}'\in \Ob(\bD(i))\:|\: \phi(i)(D_{i})=D\:\: \mbox{and}\:\:  \phi(i)(D_{i}')=D' )\}\ .$$
Note that for $i$ in $I_{D,D'}$ the objects $D_{i}$ and $D_{i}'$  are uniquely determined.
\end{enumerate}
The inclusion $\bE\to \bD$ (a morphism in $\nClincat$) induces a norm on
$\bE$. We let $\bar \bE$ be the closure of $\bE$ with respect to this norm. By construction the morphism $\phi$ factorizes over a   morphism $\bC\to \underline{\bar\bE}$ in $\Fun(I,\nCcat)$. By adjunction it induces a morphism \begin{equation}\label{dsvadsvalklkadsv}
\colim_{i\in I} \bC(i)\to \bar \bE\ . 
\end{equation} 

  \begin{lem}\label{ihergwiurergwggw}
  The morphism  \eqref{dsvadsvalklkadsv} is an isomorphism
   \end{lem}
\begin{proof}
 Using Proposition \ref{roijgeqroifqewqfewf} we see that  the colimit in $\nCcat$ can be calculated by first forming the colimit in $\nClincat$, observing that the result is a pre-$C^{*}$-category, and then applying the completion.
 Thus let $\incl:\nCcat\to \nClincat$ be the inclusion.
 It is easy to see by checking the universal property that 
 $$  \colim_{i\in I}\incl(\bC(i))\cong   \bE \ . $$ We must show that the maximal norm on $\bE$ coincides with the norm induced from $\bD$. 
 This then implies that $\Compl(\bE)\cong \bar \bE$.

  Let $f$ be a morphism in $\bE$. Then there exists $i$ in $I$ such that $f =\phi(i)( f_{i})$ for a morphism $f_{i}$ in $\bC(i)$.
We then have
 $$\|f\|_{\bD}\le \|f\|_{\max}\le \|f_{i}\|_{ \bC(i)}= \|\phi(i)(f_{i})\|_{\bD}= \|f\|_{ \bD}$$
 showing that all  inequalities are   equalities.
 Here we use that the morphism $\phi(i)$ {for each $i$ in $I$} is isometric since     it is   injective. 
  \end{proof}

   \section{Crossed products}
 
Let $G$ be a group. In this section we generalize the notion of a crossed product with $G$ from algebras to categories. 
As similar construction for additive categories can be found in \cite{bartels_reich}.
Here we describe the construction of the crossed product and its universal property in an ad-hoc manner. A more conceptual interpretation of the construction will be given  in Section \ref{wgoijorgregwergwgrwreg}.

By $BG$ we denote the category with one object $*_{BG}$ and the monoid of endomorphisms $\End_{BG}(*_{BG})\cong G$.  For a category $\cC$  the category of $G$-objects in $\cC$ is the functor category $\Fun(BG,\cC)$. Its objects are the objects of $\cC$ equipped with a $G$-action, and its morphisms are  those   morphisms in $\cC$ which are  equivariant.
 
We consider   $\bC$ in $\Fun(BG,\nClincat)$.  We use 
the notation $(g,C)\mapsto gC$ and $(g,f)\mapsto gf$ for the $G$-action on objects and morphisms of $\bC$. 
\begin{ddd}\label{wergiojwergregregwergerg}
We define the crossed product   $\bC\rtimes^{\alg}G$ of $\bC$ with $G$    as an object of  $\nClincat$ as follows:
\begin{enumerate}
\item objects: The set of objects of  $ \bC\rtimes^{\alg}G$ is the set of objects of $\bC$.
\item \label{wtiogwerfgewrgwergwerg} morphisms: For any two objects $C,C^{\prime}$ of $\bC$ we define the $\C$-vector space
$$\Hom_{\bC\rtimes^{\alg}G}(C,C^{\prime}):=\bigoplus_{g\in G} \Hom_{\bC}(C,g^{-1}C^{\prime})\ .$$
An element $f$  in the summand  $\Hom_{\bC}(C,g^{-1}C')$ will be denoted by $(f,g)$.
\item composition: For $(f,g)$ in $\Hom_{\bC\rtimes^{\alg}G}(C,C^{\prime})$ and $(f',g')$ in $\Hom_{\bC\rtimes^{\alg}G}(C',C'')$  we set  $$(f',g')\circ (f,g):=(g^{-1}f'\circ f,g'g)\ .$$ For general elements the composition is defined by linear extension. 
\item $*$-operation: We define $(f,g)^{*}:=(gf^{*},g^{-1})$.
\end{enumerate}
 \end{ddd}
 One checks by  straightforward  calculations  that $\bC\rtimes^{\alg}G$ is well-defined.
  Note that if $f:C\to C^{\prime}$ is a morphism in $\bC$ and $g$ is in $G$, then we get a morphism $(f,g):C\to gC^{\prime}$ in $\bC\rtimes^{\alg}G$.
 
 The construction of the crossed product is functorial in $\bC$ in an obvious manner. 
 Let $\phi:\bC\to \bC^{\prime}$ be a morphism in $\Fun(BG,\nClincat)$. Then we get a morphism
 $$\phi\rtimes^{\alg}G:\bC\rtimes^{\alg}G\to \bC'\rtimes^{\alg}G$$ defined
 as follows:
 \begin{enumerate}
 \item objects: The action of $\phi\rtimes^{\alg}G$ on objects  is given by the action of $\phi$ on objects.
 \item morphisms: For a morphism $f$ in $\bC$ and $g$ in $G$ we set $(\phi\rtimes^{\alg}G)(f,g):=(\phi(f),g)$.
 \end{enumerate}
 We have thus defined a functor \begin{equation}\label{vfdqeiuhfviudfv} -\rtimes^{\alg}G:\Fun(BG,\nClincat)\to \nClincat\ .  \end{equation}
 In Proposition \ref{obgfbgbgfbsfdgbsdfb9} below we will extend this functoriality from equivariant to weakly invariant functors, and  we will also incorporate  natural transformations.

  The crossed product  functor preserves unitality of objects and morphisms and therefore restricts to a functor  \begin{equation}\label{rqgkjbejkqveqvqcdscasdc} -\rtimes^{\alg}G:\Fun(BG,\Clincat)\to \Clincat)\ .\end{equation}
 
\begin{rem}
The crossed product functor $-\rtimes^{\alg}G$   preserves the full subcategories  of
algebras  $\nAlgc$ of $\nClincat$ (in the possibly non-unital case) and $\Algc$ of $\Clincat$ (in the unital case).
 The restrictions of the crossed product to these subcategories recovers the    classical definitions. \hB
\end{rem}

We have a canonical  morphism \begin{equation}\label{}\iota^{\alg}_{\bC}:\bC\to \bC\rtimes^{\alg}G \end{equation}  in $\nClincat$ which is the identity on objects and sends a morphism $f$ in $\bC$ to the morphism  $(f,e)$ in $\bC\rtimes^{\alg}G$.
If $\bC$ is unital, then   $\iota_{\bC}^{\alg}
$ is unital. 

\begin{rem}
Note that in the domain of $\iota^{\alg}_{\bC}$ we omitted to write the functor which forgets the $G$-action. 
Below we will also omit the various inclusion functors from the notation. \hB
\end{rem}

Morphisms out of a crossed product are related with the notion of a covariant representation. 
Let $\bC$ be in $\Fun(BG,\nClincat)$, and let $\bD$ be in $\nClincat$. 
\begin{ddd}\label{regiowerjhgwergwergwregrg}
A covariant representation of $\bC$ on $\bD$ is  a pair $(\rho,\pi)$ where:
\begin{enumerate}
\item $\rho:\bC\to \bD$ is a  morphism in $\nClincat$.
\item \label{qrgqfwwefqef}$\pi=(\pi(g))_{g\in G}$ is a family of unitary natural transformations $\pi(g):\rho\to g^{*}\rho$ such that
$g^{*}\pi(h)\circ \pi(g) =\pi(hg)$ for all $h,g$ in $G$. 
 \end{enumerate}
 A unitary natural transformation $\kappa:(\rho,\pi)\to (\rho',\pi')$ between covariant representations is a unitary natural transformation $\kappa:\rho\to \rho'$ such that for every $g$ in $G$ we have \begin{equation}\label{qewfjfwefewfqwfwefqw}
\pi'(g)\circ \kappa= g^{*}\kappa\circ \pi(g) .
  \end{equation}
\end{ddd}

\begin{rem}Let $(\rho,\pi)$ be a covariant representation of $\bC$ on $\bD$ and $C$ be an object of $\bC$. Then
  $\pi(g)$ is given by  a family of unitary morphisms $(\pi(g)_{C}:\rho(C)\to \rho(gC))_{C\in \Ob(\bC) }$ in $\bD$
 such that $$\pi(g)_{C'}  \rho(f)=\rho(gf) \pi(g)_{C} $$
for all morphisms $f:C\to C^{\prime}$ in $\bC$ and $g$ in $G$. 
In particular, the objects in the image of $\rho$ must have identities since Condition \ref{qrgqfwwefqef} implies $\pi(e)_{C}=\id_{\rho(C)}$ for 
all $C$ in $\bC$.  In addition we need the identity $\id_{\rho(C)}$ in order to talk about unitary morphisms out of $\rho(C)$.

A unitary natural transformation $\kappa :(\rho,\pi)\to (\rho',\pi')$ is given by a family
of unitaries $(\kappa_{C})_{C\in \Ob(\bC)}$ with $\kappa_{C}:\rho(C)\to \rho'(C)$, and the condition \eqref{qewfjfwefewfqwfwefqw} translates to the relation 
$\pi'(g)_{C} \kappa_{C}=\kappa_{g(C)} \pi(g)_{C}$ for all $g$ in $G$ and $C$ in $\bC$.  \hB
\end{rem}

\begin{rem} \label{rgiojerogergergw}Assume  that $\bC$ and $\bD$ are in $\Fun(BG,\nCcat)$. Then
one could consider the groupoid with $G$-action $\Fun_{\nCcat}(\bC,\bD)^{+}$ in $\Fun(BG,\Groupoids)$. Its objects are morphisms from $\bC$ to $\bD$  in $\nCcat$. Its  morphisms are unitary isomorphisms. The $G$-action is induced from the $G$-actions on $\bC$ and $\bD$ by conjugation. Then the groupoid of unital covariant representations of $\bC$ on $\bD$ and unitary natural transformations  is equivalent to  the groupoid of  two-categorial $G$-invariants of $\Fun_{\nCcat}(\bC,\bD)^{+}$. \hB
\end{rem}

%

Let $\bC$ be in $\Fun(BG,\nClincat)$, and let $\bD$ be in $\nClincat$.
\begin{lem}\label{gbiojwogwerewgwreer}\mbox{}
\begin{enumerate}
\item \label{reogijwioegwergwgrw} A covariant representation $(\rho,\pi)$ of $\bC$ on $\bD$  naturally induces a morphism $$\sigma:\bC\rtimes^{\alg} G\to \bD$$ in $\nClincat$. A unitary natural transformation $ (\rho,\pi)\to (\rho',\pi')$ between  covariant representations naturally induces 
 a unitary isomorphism $\sigma \to \sigma'$ between the  corresponding functors   from $ \bC\rtimes^{\alg}  G $ to $\bD$.
\item \label{reogijwioegwergwgrw1} If $\bC$ is unital, then this correspondences determines an isomorphism between the groupoids of   covariant representations $(\rho,\pi)$  of $\bC$ on $\bD$  with unital $\rho$ and unitary natural transformations,    and unital morphisms   $\bC\rtimes^{\alg} G\to \bD$ and unitary isomorphisms. \end{enumerate}
\end{lem}
\begin{proof}
Let the covariant representation $(\rho,\pi)$ be given.
  Then we define  the associated morphism $\sigma:\bC\rtimes^{\alg}G\to \bD$ as follows:
  \begin{enumerate}
  \item objects: The action of $\sigma$ on objects is given by the action of $\rho$ on objects.
  \item morphisms: The action of $\sigma$ on morphisms is determined by linearity and \begin{equation}\label{wfqwefwefwefqwfqrevevervewfqwef} \sigma(f,g):=\pi(g)_{C'}  \rho(f):\rho(C)\to \rho(gC^{\prime})\end{equation} 
  for all $g$ in $G$ and morphisms $f:C\to C^{\prime}$ in $\bC$. 
  \end{enumerate}
  One easily checks that $\sigma$ is compatible with the composition and the involution. 
  
 If   $\kappa:(\rho,\pi)\to (\rho',\pi')$  is  a  unitary natural transformation  between covariant representations  given by a family $(\kappa_{C})_{C\in \Ob(\bC)}$, then the same family    can be interpreted in the canonical way as a natural transformation $  \sigma\to \sigma'$.     This finishes the proof of Assertion \ref{reogijwioegwergwgrw}.

In order to show Assertion \ref{reogijwioegwergwgrw1} we first observe that  
if $\bC$ and $\rho$ are unital, then the functor $\sigma$ constructed above  is unital.

Assume now that $\bC$ is unital and let  $\sigma:\bC\rtimes^{\alg}G\to \bD$ be a  given unital functor.
  Then we define the unital functor $\rho:=\sigma\circ\iota_{\bC}^{\alg}:\bC\to \bD$.
 Furthermore, for every $g$ in $G$ and object $C$ in $\bC$ we define \begin{equation}\label{wefvbekjnvvjfrsfvfdv} \pi(g)_{C}:=\sigma(\id_{C},g) \ .\end{equation}
  Then
$(\rho,\pi)$ is the desired covariant representation. One checks that $(\rho,\pi)$ satisfies the Condition \ref{regiowerjhgwergwergwregrg}.\ref{qrgqfwwefqef}, and that  the functor associated by \ref{reogijwioegwergwgrw}.
to this covariant representation is the original $\sigma$.

Let $\kappa:\sigma\to \sigma'$ be a unitary natural transformation. By restriction along $\iota_{\bC}^{\alg}$ we can interpret $\kappa$ as a natural transformation $\kappa:\rho \to \rho'$.  Naturality of  $\kappa$ and  \eqref{wefvbekjnvvjfrsfvfdv} together  imply the relations \eqref{qewfjfwefewfqwfwefqw}.
 \end{proof}

In the case of $C^{*}$-categories it is natural to consider a completed version of the crossed product. Our categorical perspective dictates to consider the completion with respect to the maximal norm.
The following lemma ensures that this completion exists. 
 Recall the Definition   \ref{ruighqregqrgqfewfqqf} of a pre-$C^{*}$-category.
\begin{lem} \label{rgioqwfqwefqefe}
If
 $\bC$  is 
in $\Fun(BG,\nCcat)$, then  $\bC\rtimes^{\alg}G$ is a pre-$C^{*}$-category. \end{lem}
  \begin{proof} 
  Recall the Definition  \ref{egergwefqwefqwfqewf}  of the maximal semi-norm.  
 We first show that for every morphism $f$ in $\bC$ and $g$ in $G$ 
 we have  \begin{equation}\label{} \|(f,g)\|_{_{\max}}\le \|f\|_{\bC}\ .\end{equation}  
 
 Let $A$ be a $C^{*}$-algebra and 
  $\lambda:\bC\rtimes^{\alg}G\to A$ be a  morphism in $\nClincat$.
 Then the composition $\lambda\circ \iota_{\bC}^{\alg}:\bC\to A$ is a functor  between $C^{*}$-categories. This implies
 $$ \|\lambda(f,e)\|_{A}= \|\lambda(\iota_{\bC}^{\alg}(f))\|_{A}\le \|f\|_{\bC}\ .$$
   We now have
$$ \|\lambda(f,g)\|^{2}_{A}  =\|\lambda(f,g)^{*}\lambda(f,g)\|_{A}= \|\lambda(f^{*}f,e)\|_{A}  \le \|f^{*}f\|_{\bC} = \|f\|^{2}_{\bC}\ . $$
   Since $A$ and $\lambda$ are arbitrary this implies that $\|(f,g)\|_{\max}\le \|f\|_{\bC}$.

Since every  morphism of $\bC\rtimes^{\alg}G$ is a finite linear combination of elements of the form $(f,g)$ this implies that
$\|-\|_{\max}$ is finite. Hence $\bC\rtimes^{\alg}G$ is a pre-$C^{*}$-category. 
 \end{proof}

In view of Lemma \ref{rgioqwfqwefqefe} we can restrict the crossed product functor to a functor
$$-\rtimes^{\alg}:\Fun(BG,\nCcat)\to \npClincat\ .$$

Assume that $\bC$ is 
in $\Fun(BG,\nCcat)$. Recall the completion functor $\Compl$ from \eqref{oidfhbuihsfuibfefefesdfbsfb}.
\begin{ddd}\label{feivoevvewfvfevfdsv}
We define the crossed product  for $C^{*}$-categories  by $$\bC\rtimes G:=\Compl(\bC\rtimes^{\alg}G)\ .$$
\end{ddd}

Since the crossed-product for $C^{*}$-categories is obtained  by applying the algebraic crossed product functor \eqref{vfdqeiuhfviudfv}
(which sends $C^{*}$-categories to pre-$C^{*}$-categories by Lemma \ref{oidfhbuihsfuibfefefesdfbsfb}) and the completion functor \eqref{oidfhbuihsfuibfefefesdfbsfb} it is clear that we have defined a functor
\begin{equation}\label{efkjbjkdfviuqr3} -\rtimes G:\Fun(BG,\nCcat)\to \nCcat\ . \end{equation}
It again restricts to a functor
 \begin{equation}\label{efkjbjkdfviuqr3f} -\rtimes G:\Fun(BG,\Ccat)\to \Ccat\ . \end{equation}
We refer to Proposition \ref{obgfbgbgfbsfdgbsdfb9}  for an extension of  this functoriality from equivariant to weakly invariant functors, and  the  incorporation of   natural transformations.

We define the natural morphism  \begin{equation}\label{sfdvsq34gervv}
\iota_{\bC}:\bC\stackrel{\iota_{\bC}^{\alg}}{\to}\bC\rtimes^{\alg}G\stackrel{\eqref{qwefkjiqkwefqwefqewf}}{\to} \bC\rtimes G  
\end{equation} in $\nCcat$.
We will see later in Corollary \ref{wethgiowrthrthwtgregrweg} that $\iota_{\bC}$ is isometric.
 
 By construction the morphism $\iota_{\bC}$ has an obvious universal property.
 Let $\bC$ be in $\Fun(BG,\nCcat)$, and let $\bD$ be in $\nCcat$.
\begin{kor}\label{qergiowegwregergewgrg}\mbox{}
\begin{enumerate}
\item \label{oighgioergwergwerg} A covariant representation $(\rho,\pi)$ of $\bC$ on $\bD$  naturally induces a morphism $\sigma:\bC\rtimes  G\to \bD$. A unitary natural transformation $\kappa:(\rho,\pi)\to (\rho',\pi')$ between  covariant representations naturally induces 
 a unitary isomorphism $\sigma \to \sigma'$ between the  corresponding  morphisms  from $\bC\rtimes G$ to $ \bD$.
\item \label{trhiorghrethtrhrhetr} If $\bC$ is unital, then the correspondence between  covariant representations $(\pi,\rho)$   of $\bC$ on $\bD$ with unital $\rho$ and  unital morphisms $\bC\rtimes  G\to \bD$ is bijective.  If $\bC$ is unital, then this correspondences determines an isomorphism between the groupoids of   covariant representations $(\rho,\pi)$  of $\bC$ on $\bD$  with unital $\rho$ and unitary natural transformations,    and unital morphisms  $\bC\rtimes G\to \bD$ and unitary isomorphisms. 
\end{enumerate}
\end{kor}
\begin{proof}
This follows from Lemma \ref{gbiojwogwerewgwreer} and the universal property of the completion.
\end{proof}

We can apply  Definition \ref{feivoevvewfvfevfdsv} in the case where $\bC$ is in $\Fun(BG,\nCalg)$, i.e., a possibly non-unital $C^{*}$-algebra with an action of $G$. In this case we also have the classical maximal crossed product $\bC\rtimes^{C^{*}}G$ \cite[Lem. 2.27]{williams}, \cite[Def. 2.3.3]{celg}, see Definition \ref{qreugiqwgrefwqewfqewf}.
It is defined as a completion of $\bC\rtimes^{\alg}G$ with respect to a norm $\|-\|_{C^{*}}$ (see \eqref{dascadscdecsdcsadcasdcasdc} below) obtained as a supremum over covariant representations in the sense of Definition \ref{wrthiowgregwergwreg} below.  We clearly have
an inequality \begin{equation}\label{wefqwfwqefwef} \|-\|_{C^{*}}\le \|-\|_{\max} \end{equation}  and  therefore a natural homomorphism of $C^{*}$-algebras \begin{equation}\label{etbwerervev} \bC\rtimes G\to \bC\rtimes^{C^{*}}G \end{equation}
The following proposition says that the two definitions of crossed products actually coincide.

Let $\bC$ be in $\Fun(BG,\nCalg)$.
 \begin{prop}\label{egwergergwrgw}
 The canonical morphism  $\bC\rtimes G\to \bC\rtimes^{C^{*}}G$ is an isomorphism.
 \end{prop}
The proof of this proposition will be given at the end the present section. The following material serves as a preparation.

Let $\bC$ be in $\Fun(BG,\nCalg)$, and let $\bD$ be in $\nCalg$. By $M(\bD)$ we denote the multiplier algebra of $\bD$.

 \begin{ddd}\label{wrthiowgregwergwreg}
A covariant representation of $\bC$ on $\bD$ is a pair $(\rho,\mu)$ of a  homomorphism $\rho:\bC\to \bD$ in $\nCalg$ and a homomorphism of groups $G\to U(M(\bD))$ such that for every $g$ in $G$ and  element  $f$ of the algebra $\bC$ we have \begin{equation}\label{werglknkeglergwegwgwrg}
\mu(g) \rho(f)\mu(g^{-1})=\rho(gf)\ .
\end{equation}
\end{ddd}
Note that  elements of the algebra  $\bC$ are morphisms of the category $\bC$.

\begin{rem}  Note that the notion of a covariant representation in the case of $C^{*}$-algebras is more general then the one given in
Definition \ref{regiowerjhgwergwergwregrg}.  Both definitions  coincide if $\pi$ actually takes values in $U(\bD)$ (e.g. if $\bD$ is unital). 
But we still have the analogue of Corollary  \ref{qergiowegwregergewgrg} for these more general covariant representations, see Lemma \ref{gbiojwogwerewgweeereEEEer}.  \hB
  \end{rem}
%
%
%
%
%
%
%
%
%

Recall that a homomorphism $\rho:A\to B$ between $C^{*}$-algebras is called essential if the  linear span of $\rho(A)B$ is dense in $B$. An essential homomorphism has a unique extension $M(\rho):M(A)\to M(B)$ to the multiplier algebras.

Let $\bC$ be in $\Fun(BG,\nCalg)$, and let $\bD$ be in $\nCalg$.
\begin{lem}\label{gbiojwogwerewgweeereEEEer}\mbox{}
\begin{enumerate}
\item \label{qeojqoergefewfqwef} A  covariant representations $(\rho,\mu)$ (in the sense of Definition \ref{wrthiowgregwergwreg}) induces a homomorphism $\sigma:\bC\rtimes  G\to \bD$ in a canonical way. 
\item\label{efiijfoiofjeojgoifewfewfqwef} If $ \sigma:\bC\rtimes  G\to \bD$ is an essential homomorphism, then it is induced from a covariant representations $(\rho,\mu)$ 
 (in the sense of Definition \ref{wrthiowgregwergwreg}) as in \ref{qeojqoergefewfqwef}.
\end{enumerate}
%
%
 \end{lem}
\begin{proof}
Let $(\rho,\mu)$ be a covariant representation of $\bC$ on $\bD$ in the sense of Definition \ref{wrthiowgregwergwreg}.  Then $\sigma:\bC\rtimes G\to \bD$ is determined by linearity, continuity, and the formula
$$\sigma(f,g):=\mu(g) \rho(f)$$ for all $g$ in $G$ and elements $f$ of the algebra $\bC$.  

Vice versa, assume that $\sigma:\bC\rtimes G\to \bD$ is given such that $\sigma$ is essential. We must reconstruct a covariant representation $(\rho,\mu)$. 
If $\bC$ and $\sigma$ were  unital, then we could appeal to Corollary \ref{qergiowegwregergewgrg}.\ref{trhiorghrethtrhrhetr}. But the general case is a little more involved.

 We set $\rho:=\sigma\circ \iota_{\bC}:\bC\to \bD$, where $\iota_{\bC}$ is as in \eqref{sfdvsq34gervv}.

In order to construct $\mu$ we first define a homomorphism $\nu:G\to M(\bC\rtimes G)$. 
To do so  we interpret  elements of the multiplier algebra as double centralizers \cite{busby}.

For every $g$ in $G$ we define   the  double centralizer   $(L(g),R(g))$ on $\bC\rtimes G $. We start with the definition of $L(g)$ and $R(g)$ as linear maps on $\bC\rtimes^{\alg}G$. They are then determined by 
the formulas 
\begin{equation}\label{wfqwefwefqwfqewfqewf} L(g)( f,h) := (f,gh)\ , \quad R(g)( f,h ):= (g^{-1}f,hg)\end{equation}
for all elements $f$ of $\bC$ and $h$ in en $G$.  
One easily verifies the relation 
\begin{equation}\label{wfqwefwefqwfqewfqewdwdwdwwdwdwf} R(g) (f',h') (f,h)= (f',h') L(g)( f,h) \end{equation}
for all $f,f^{\prime}$ in $\bC$ and $h,h'$ in $G$. 
Next we show that $R(g)$ and $L(g)$ extend by continuity to $\bC\rtimes G $.
One calculates that
$$  (L(g)(f',h'))^{*}L(g)(f,h)=(f',h')^{*}(f,h)$$
for all $(f,h)$, $(f',h')$ in $\bC\rtimes^{\alg} G $.
This implies that 
\begin{equation}\label{wfqwefwefqwfefefefefefffeffefefeffqewfqewdwdwdwwdwdwf}  \|L(g)(a)\|^{2}_{\max}=\|L(g)(a)^{*} L(g)(a)\|_{\max}=\|a^{*}a\|_{\max}=\|a\|_{\max}^{2}\end{equation}
for every $a$ in $\bC\rtimes^{\alg} G $.
Hence $L(g)$ extends by continuity to an isometry of $\bC\rtimes  G$.
We  next note that for $a$ in $\bC\rtimes^{\alg} G$ we have 
\begin{equation}\label{wfqwefwefqwfeffeffefefeffqewfqewdwdwdwwdwdwf} \|R(g)(a)\|_{\max}=\sup_{b\in \bC\rtimes^{\alg} G, \|b\|_{\max}\le 1} \| R(g)(a)b\|_{\max}\ .\end{equation}
In fact, the inequality $\le$ follows from the sub-multiplicativity of the maximal norm. In order to see the equality (in the non-trivial case that $\|R(g)(a)\|_{\max}\not=0$) we insert $b:=\|R(g)(a)\|_{\max}^{-1} R(g)(a)^{*}$
and use the $C^{*}$-identity.  

Applying the Relations \eqref{wfqwefwefqwfqewfqewdwdwdwwdwdwf} and  \eqref{wfqwefwefqwfefefefefefffeffefefeffqewfqewdwdwdwwdwdwf} to the right-hand side of \eqref{wfqwefwefqwfeffeffefefeffqewfqewdwdwdwwdwdwf} gives
$$\hspace{-0.5cm}\|R(g)(a)\|_{\max} =
\sup_{b\in \bC\rtimes^{\alg} G, \|b\|_{\max}\le 1} \| a L(g)(b)\|_{\max} \le 
\|a\|_{\max}  \sup_{b\in \bC\rtimes^{\alg} G, \|b\|_{\max}\le 1} \|L(g)(b)\|_{\max} \le 
\|a\|_{\max} \ .
$$
Hence also $R(g)$ extends by continuity to $\bC\rtimes  G$.
Consequently, the pair $(L(g),R(g))$ determines a multiplier $\nu(g)$ in $M(\bC\rtimes G)$ such that
  \begin{equation}\label{eqwfqwefewfwf42reee1}\nu(g)a=L(g)(a)\ ,\quad a\nu(g)=R(g)(a)  \end{equation}
  for arbitrary $a$ in $\bC\rtimes G$. 
Using the  formulas \eqref{wfqwefwefqwfqewfqewf}  and \eqref{eqwfqwefewfwf42reee1} we next check     that
$\nu(g)$ is unitary for every $g$ in $G$.  Let $(f,h)$ be in   $\bC\rtimes^{\alg} G $. Then we calculate
\begin{align*}
\nu(g)^{*}\nu(g)(f,h)&=\nu(g)^{*} L(g)(f,h)= \nu(g)^{*}(f,gh)\\
=&((f,gh)^{*}\nu(g))^{*}=((ghf^{*},h^{-1}g^{-1})\nu(g))^{*}=R(g)(ghf^{*},h^{-1}g^{-1})^{*}\\
=&(hf^{*},h^{-1})^{*}=(f,h)\ .
\end{align*}
This implies that $\nu(g)\nu(g)^{*}=1$.
Similarly we check that $\nu(g)\nu(g)^{*}=1$. For $g,g'$ in $G$ and $(f,h)$  in   $\bC\rtimes^{\alg} G $ we have
$$\nu(g)\nu(g')(f,h)=\nu(g) L(g')(f,h)=L(g)( f,gh)=(f,gg'h)=\nu(gg')(f,h)\ .$$
This implies that 
the map $\nu:G\to U(M(\bC\rtimes G))$ is a homomorphism of groups.

We now note   that $(\iota_{\bC},\nu)$ is a covariant representation of $\bC$ on $\bC\rtimes G$ in the sense of Definition \ref{wrthiowgregwergwreg}.

Since we assume that $\sigma$ is essential we can consider the extension  $M(\sigma):M(\bC\rtimes G)\to M(\bD)$  of $\sigma$  to the multiplier algebras.
Then we set $\mu:=M(\sigma)\circ \nu:G\to U(M(\bD))$.
The  pair $(\rho ,\mu)$ is a covariant representation   in the sense of Definition \ref{wrthiowgregwergwreg}.
The homomorphism $\bC\rtimes G\to \bD$ associated to  this covariant representation is clearly $\sigma$.
 \end{proof}

Let $\bC$ be in $\Fun(BG,\nCalg)$.    We define the $C^{*}$-semi-norm of an element $x$ of $\bC\rtimes^{\alg}G$ as
\begin{equation}\label{dascadscdecsdcsadcasdcasdc}\|x\|_{C^{*}}:=\sup_{(\rho,\mu)} \|\sigma(x)\|_{\bD} \ , \end{equation} 
where $(\rho,\mu)$ runs over all covariant representations of $\bC$ on $\bD$   in the sense of Definition \ref{wrthiowgregwergwreg}, and $\sigma:\bC\rtimes G\to \bD$ is the homomorphism associated to $(\rho,\mu)$ by \ref{gbiojwogwerewgweeereEEEer}.\ref{qeojqoergefewfqwef}.

\begin{ddd}\label{qreugiqwgrefwqewfqewf}
We define $\bC\rtimes^{C^{*}}G$ as the completion of $\bC\rtimes^{\alg} G$ with respect to the norm  $\|-\|_{C^{*}}$.
\end{ddd}

 \begin{proof}[Proof of Proposition \ref{egwergergwrgw}]
We must show the inequality $$\|-\|_{C^{*}}\ge \|-\|_{\max}$$ on
$\bC\rtimes^{\alg}G$. The identity $\id:\bC\rtimes G\to \bC\rtimes G$  is esssential and therefore provides by Lemma \ref{gbiojwogwerewgweeereEEEer}.\ref{efiijfoiofjeojgoifewfewfqwef}  
a covariant representation $(\rho,\mu)$  of $\bC$ on $\bC\rtimes G$  in the sense of Definition \ref{feivoevvewfvfevfdsv} .
In view of \eqref{dascadscdecsdcsadcasdcasdc} we see that $\|-\|_{C^{*}} \ge \|-\|_{\max} $. Because of \eqref{wefqwfwqefwef}
we conclude that $\|-\|_{C^{*}}=\|-\|_{\max}$. 
 \end{proof}

The following Lemma is a special case of   Corollary \ref{wethgiowrthrthwtgregrweg}, but it is actually  used in its proof and therefore needs an independent verification.
\begin{lem}\label{regiuwergrerwegweg}
If $\bC$ is in  $\Fun(BG,\nCalg)$,  then $\iota_{\bC}:\bC\to \bC\rtimes G$ is isometric.
\end{lem}
\begin{proof}
The representation of
$\bC \rtimes G$ on the  $\bC$-Hilbert $C^{*}$-module
$L^{2}(G,\bC)$ induces the reduced norm $\|-\|_{r}$ on $\bC\rtimes G$. 
One checks that   for $c$ in $\bC$ we have $\|c\|_{\bC}=\|\iota_{\bC}(c)\|_{r}$ since $\iota_{\bC}(c)$ acts on $L^{2}(G,\bC)$ as the multiplication operator with the constant function with value $c$.  We furthermore have a chain of inequalities
 $$\|c\|_{\bC}=\|\iota_{\bC}(c) \|_{r}\le \|\iota_{\bC}(c)\|_{\bC\rtimes G}\le \|c\|_{\bC}\ .$$
Consequently, the inequalities are equalities and therefore $\iota_{\bC}$ is isometric.
\end{proof}

 \section{From categories to algebras - the functor $A$}
 
 Sometimes one can show facts for $C^{*}$-categories using using known facts about $C^{*}$-algebras.  In this direction the functor 
 $$A:\nCcatinj \to \nCalg$$ introduced by    e.g. in \cite[Sec. 3]{joachimcat}
 is a helpful tool. The main application of this functor in the present section is Corollary \ref{wethgiowrthrthwtgregrweg} saying that the canonical functor $\bC \to \bC\rtimes G$ in  \eqref{sfdvsq34gervv} is an isometry.
 On the way we show in Theorem \ref{qrioqwfewfewfewfqef} that $A$ commutes with forming  crossed  products.

We let $\nClincatinj$ be the wide subcategory of $\nClincat$ whose morphisms are  only those  functors which are injective on objects.  

 \begin{ddd}\label{hiuwegregwergweg}
  The functor $A^{\alg}:\nClincatinj\to \nAlgc$ is defined as follows:
  \begin{enumerate}
  \item objects:  \begin{enumerate}
  \item The underlying $\C$-vector space of $A^{\alg}(\bC)$ is \begin{equation}\label{qwefj1ij4oi3fwef}
A^{\alg}(\bC):=\bigoplus_{C,C^{\prime}\in \Ob(\bC)}\Hom_{\bC}(C,C^{\prime})\ .
\end{equation} A morphism $f:C\to C^{\prime}$ in $\bC$ gives rise to an element  $[f]$  in $A^{\alg}(\bC)$.
\item The product in $A^{\alg}(\bC)$ is determined 
by linearity and
 $$ [f'][f]:=\left\{\begin{array}{cc}[f'\circ f]& C''=C'\\0&else \end{array}\right.$$
 for all pairs of morphisms $f:C\to C'$ and $f^{\prime}:C''\to C'''$ in $\bC$
\item The involution on $A^{\alg}(\bC)$ is determined by $[f]^{*}=[f^{*}]$.
  \end{enumerate}
  \item\label{igjogewgwergwerg} morphisms: The functor $A^{\alg}$ sends a morphism $\phi:\bC\to \bC^{\prime}$  in   $\nClincatinj$  to the homomorphism $A^{\alg}(\phi):A^{\alg}(\bC)\to A^{\alg}(\bC^{\prime})$  sending $[f]$ to $[\phi(f)]$ for every morphism $f$ in $\bC$. 
  \end{enumerate}
\end{ddd}
 
 \begin{rem}
 Note that in general $A^{\alg}(\phi)$   in \ref{igjogewgwergwerg}.  is only well-defined if $\phi$  is  injective on objects. Therefore  we need the restriction to 
 the subcategory of functors which are injective on objects. \hB
 \end{rem}
 
 We have a natural morphism $$\rho_{\bC}^{\alg}:\bC\to A^{\alg}(\bC)$$ which is uniquely determined
 by the condition that its sends a morphism $f$ of $\bC$ to the element $[f]$ in $A^{\alg}(\bC)$.

 \begin{lem}\label{wegiowegerregwergwerg}
 The morphism $\rho_{\bC}^{\alg}:\bC\to A^{\alg}(\bC)$ is initial for morphisms $\sigma:\bC\to A$ in $\nClincat$ from $\bC$ to $A$ in $\nAlgc$ with the property that $$\sigma(f')\sigma(f):=\left\{\begin{array}{cc}\sigma(f'\circ f)& \dom(f')=\codom(f)\\0&else \end{array}\right.\ .$$ 
 \end{lem}
 \begin{proof}
 This is obvious from the definition.
 \end{proof}

 We let $\nCcatinj$ be the full subcategory of $\nClincatinj$ consisting of  possibly non-unital $C^{*}$-categories.
Recall the definition \eqref{wqrefonwiefjqwefewqfw}  of the category
 $\npAlgc$   of pre-$C^{*}$-algebras.

 The following has been shown in the proof of \cite[Lemma 3.6]{joachimcat}.
 \begin{lem}
 If $\bC$ is in $\nCcatinj$, then $A^{\alg}(\bC)$ is a pre-$C^{*}$-algebra.
 \end{lem}
\begin{proof}
Every element of $A^{\alg}(\bC)$ is a finite linear combination of elements of the form $ [f]$ for morphisms $f$ in $\bC$.

 Assume that $f:C\to C^{\prime}$ is a morphism in $\bC$. It suffices to show that $\| [f]\|_{\max}$ is finite.  Note that  $f^{*}f$ is an element of the $C^{*}$-algebra $\End_{\bC}(C)$. 
Consequently we have the middle inequality in the following chain $$\| [f]\|_{\max}^{2}=\|[f^{*}f]\|_{\max}\le \|f^{*}f\|_{\End_{\bC}(C)}=\|f\|^{2}_{\bC}\ .$$
 \end{proof}

 \begin{ddd}\label{zhioerhrthtrhehth}
We define the functor $A:\nCcatinj\to \nCalg$ as the composition
$$A:\nCcatinj\stackrel{A^{\alg}}{\to} \npAlgc\stackrel{\Compl, \eqref{ervevwevefvdsfvsdfvsdfv}}{\to} \nCalg\ .$$
\end{ddd}

We have a canonical morphism \begin{equation}\label{eqwfqewf124rwrefqfew}
\rho_{\bC}:\bC\stackrel{\rho_{\bC}^{\alg}}{\to} A^{\alg}(\bC)\stackrel{\eqref{qwefkjiqkwefqwefqewf}}{\to} A(\bC)\ .
\end{equation}
 
 \begin{lem}\label{egioerjgoergwegwergwregrwe} The morphism $\rho_{\bC} :\bC\to A (\bC)$ is initial for morphisms $\sigma:\bC\to A$ in $\nCcat$ from $\bC$ to $A$ in $\nCalg$ with the property that $$\sigma(f')\sigma(f):=\left\{\begin{array}{cc}\sigma(f'\circ f)& \dom(f')=\codom(f)\\0&else \end{array}\right.\ .$$  
 \end{lem}
 \begin{proof}
 This follows from Lemma  \ref{wegiowegerregwergwerg} and the universal property of the completion functor.
 \end{proof}

  \begin{lem}\label{ergtegwergwergwrg}
 The morphism $\rho_{\bC}:\bC\to A(\bC)$ is isometric.
 \end{lem}
\begin{proof}
This has been observed in the proof of \cite[Lemma 3.6]{joachimcat}.  
As this fact is crucial for later applications we  recall the argument.

Let $C$ be an object of $\bC$. We  form the $\End_{\bC}(C)$-right module $$M^{\alg}_{C}:=\bigoplus_{C'\in \Ob(C)} \Hom_{\bC}(C,C' )\ .$$
A morphism $h:C\to C^{\prime}$ gives rise to an element $[h]$ in $M^{\alg}_{C}$. The $C^{*}$-algebra 
 $\End_{\bC}(C)$ acts by right composition such that $[h][f]=[h\circ f]$ for every $f:C\to C$. Furthermore,
the algebra $A^{\alg}(\bC)$ acts from the left  on $M_{C}^{\alg}$ by matrix multiplication such that
$$[f][h]=\left\{\begin{array}{cc}[f\circ h]&C'=C''\\0&else\end{array}\right. $$
for all morphisms $f:C''\to C'''$ in $\bC$.

  We define the $ \End_{\bC}(C)$-valued scalar product on $M_{C}^{\alg}$ such that
  $$\langle [g],[h]\rangle=\left\{\begin{array}{cc} g^{*}\circ h& C'=C'' \\0&else\end{array}\right. $$
  for all $g:C\to C''$ and $h:C\to C^{\prime}$.

   We then let $M_{C}$ be $\End_{\bC}(C)$ Hilbert-$C^{*}$-module  given by the completion of $M^{\alg}_{C}$
   with respect to the norm  $\|-\|_{M_{C}}$ defined by the scalar product.
 The representation
$A^{\alg}(\bC)\to \End_{\End_{\bC}(C)}(M)$ yields a $C^{*}$-norm $\|-\|_{C}$ on $A^{\alg}(\bC)$. 

We claim that $$ \|[f]\|_{C}=\|[f]\|_{\max}$$ for every $f$ in $\End_{\bC}(C)$. 
We first observe that  \begin{eqnarray*} 
\|[f]\|_{C}&=&\sup_{h\in M_{\bC}^{\alg}, \|h\|_{M_{C}}=1} \|[f]h\|_{M_{C}}\\
&\stackrel{!}{=}&\sup_{h\in  \Hom_{\bC}(C,C), \|h\|_{\End_{\bC}(C)}=1} \|f\circ h\|_{\bC}\\&\stackrel{!!}{=}& \| f \|_{\bC}\ .
\end{eqnarray*}
For the equality marked by $!$ we use 
 that left-multiplication by $[f]$ annihilates all summands of $M_{C}^{\alg}$ except the one with index $C$. Furthermore, for the equality marked by $!!$
  we use that the canonical morphism of a $C^{*}$-algebra into its multiplier algebra is isometric. 
 Since $\|f\|_{\bC}=\|[f]\|_{C}\le \|[f]\|_{\max}$ we can conclude that the homomorphism 
   $\End_{\bC}(C)\to A(\bC)$ of $C^{*}$-algebras is injective. It is therefore isometric which shows the claim. 

 Since we can choose  the object $C$ arbitrary 
we conclude that $\rho_{\bC}$ is isometric on the endomorphisms of every object of $\bC$. For $f:C\to C^{\prime}$ we then get
$$\|\rho_{\bC}(f)\|_{\max}^{2}=\|\rho_{\bC}(f^{*}f)\|_{\max}=\|f^{*}f\|_{\bC}=\|f\|_{\bC}^{2}\ .$$
 \end{proof}

\begin{lem}\label{ergkljewrogwergwregf}\mbox{}\begin{enumerate}
\item \label{weirgowergwergwerg}
The functor $A:\nCcatinj\to \nCalg$ preserves isometric inclusions.
\item  \label{weirgowergwergwerg1} 
For $\bC$ in $\nCcat$ every injective homomorphism $A^{\alg}(\bC)\to B$ into a $C^{*}$-algebra $B$ extends to an isometric homomorphism $A(\bC)\to B$.
\end{enumerate}
\end{lem}
\begin{proof}
We start with the proof of Assertion \ref{weirgowergwergwerg}. 
Let $i:\bC\to \bD$ be an isometric inclusion in $\nCcatinj$.   Then we must show that 
 the morphism $A(\bC)\to A(\bD)$  in $\nCalg$ is isometric. To this end we consider the diagram 
$$\xymatrix{\bC\ar[r]^{i}\ar[d]^{\rho_{\bC}}&\bD\ar[d]^{\rho_{\bD}}\\A(\bC)\ar[r]^{A(i)}&A(\bD)}\ $$
By Lemma \ref{ergtegwergwergwrg} the vertical morphisms are isometric. 
Furthermore, $i$ is isometric by assumption.

Let $\bC^{\prime}$ be the full subcategory of $\bC$ on some finite set of objects of $\bC$. Then we get a morphism $A^{\alg}(\bC')\to A(\bC)$ whose image is a closed subalgebra (since the morphism $\rho_{\bC}$ is isometric). Hence the Banach space  topology on  $A^{\alg}(\bC')$ induced from the Banach space topology of the morphism spaces of $\bC'$ coincides with the topology induced from $A(\bC)$.
 We let $\bD^{\prime}$ be the full subcategory of $\bD$ on the objects $i(\Ob(\bC'))$. Then we have a diagram
$$\xymatrix{A(\bC')\ar[r]^{A(i')}\ar[d]&A(\bD')\ar[d]\\A(\bC)\ar[r]^{A(i)}&A(\bD)}\ .$$
The map $A(i')$ is a closed embedding for the Banach space topologies induced from $\bC^{\prime}$ and $\bD^{\prime}$, and therefore also a closed embedding for the topologies induced from the vertical arrows.  $A(i')$  is  furthermore  morphism of $C^{*}$-algebras with respect to the norms induced by the vertical arrows and hence an isometry with respect to these norms. 

Since $A(\bC)$ and $A(\bD)$ are by definition the closures of the union of the images of the vertical maps for all choices of $\bC'$
we conclude that $A(i)$ is also an isometry. 

We now show  Assertion \ref{weirgowergwergwerg1}.
 Let $A^{\alg}(\bC)\to B$ be an injective homomorphism into some $C^{*}$-algebra. Then for every full subcategory $\bC'$ of $\bC$ with finitely many objects $A(\bC')\to A^{\alg}(\bC)\to B$ is injective and hence isometric.
 As in the end of the argument above this implies that $A(\bC)\to B$ is isometric, too.
 \end{proof}

We next show that the functor $A$  in Definition \ref{zhioerhrthtrhehth} commutes with crossed products.

  We  assume that $\bC$ in $\Fun(BG,\nCcat)$.
Since $G$ acts by invertible functors we  actually have   $\bC$ in $\Fun(BG,\nCcatinj)$.
By functoriality of $A$ we can then consider  $A(\bC)$ in $\Fun(BG,\nCalg)$. 

We start with the construction  of a covariant representation which will eventually induce the comparison map \eqref{ewfqewfqewfqwefqwfeqwfqewfqewf}.
 We have a  morphism  $$\iota_{\bC}:\bC\to \bC\rtimes G$$ in $\nCcatinj$, see \eqref{sfdvsq34gervv}. By functoriality of $A$ it induces a morphism $$A(\iota_{\bC}):A(\bC)\to A( \bC\rtimes G)\ .$$

 \begin{lem}
We have a canonical homomorphism $\pi_{\bC}:G\to U(M(A(\bC\rtimes G)))$ such that $(A(\iota_{\bC}) ,\pi_{\bC}     )$ is a covariant representation of $A(\bC)$ on $A(\bC\rtimes G)$ in the sense of Definition \ref{wrthiowgregwergwreg}. 
\end{lem}
\begin{proof} We repeat the corresponding argument from the proof of Lemma \ref{gbiojwogwerewgweeereEEEer}.\ref{efiijfoiofjeojgoifewfewfqwef}.
For every $g$ in $G$ we define   the  double centralizer   $(L(g),R(g))$ on $A(\bC\rtimes G) $. We start with the definition of $L(g)$ and $R(g)$ as linear maps on $A^{\alg}(\bC\rtimes^{\alg} G)$. They are then determined by 
the formulas 
\begin{equation}\label{wfqwefwefqwfqewfqewf4r4} L(g)([ f,h]):= [f,gh]\ , \quad R(g)([f,h ]):= [g^{-1}f,hg]\end{equation}
for all morphisms $f$ of $\bC$ and $h$ in $G$. Then 
one easily verifies the relation 
\begin{equation}\label{wfqwefwefqwfqewvfvdfvdvdfvfqewf4r4} R(g) ([f',h']) [f,h]= [f',h'] L(g)([ f,h])\end{equation}
for all morphisms $f,f^{\prime}$ in $\bC$ and $h,h'$ in $G$. 
 One calculates that \begin{equation}\label{fdspvojsdopfvsfdvsdfvsfdvsfdv}
 (L(g)([f',h']))^{*}L(g)([f,h])=[f',h']^{*}[f,h]
\end{equation} 
  for every $f,f'$ in $\bC$ and $h,h'$ in $G$.  We now show that 
   $R(g)$ and $L(g)$ extend by continuity to $A(\bC\rtimes G) $.  
We do this in two steps. We first extend them  to $A^{\alg}(\bC\rtimes G)$, and then to $A(\bC\rtimes G) $.  For the first step    
we observe that $L(g)$   maps the summand $\Hom_{\bC\rtimes^{\alg} G}(C,C^{\prime})$ of 
$A^{\alg}(\bC\rtimes^{\alg} G)$ (see \eqref{qwefj1ij4oi3fwef}) to the summand $ \Hom_{\bC\rtimes^{\alg} G}(C,gC^{\prime})$. 
  The equation \eqref{fdspvojsdopfvsfdvsdfvsfdvsfdv}   then   implies
  for arbitrary $a$ in $\Hom_{\bC\rtimes^{\alg} G}(C,C^{\prime})$ that 
$$\|L(g)([a])\|_{\Hom_{\bC\rtimes  G}(C,gC^{\prime})} = \|a\|_{\Hom_{\bC\rtimes  G}(C,C^{\prime})}
\ .$$
This provides the continuous extension of $L(g)$ to $A^{\alg}(\bC\rtimes G)$.
Again using  the equation  \eqref{fdspvojsdopfvsfdvsdfvsfdvsfdv} we now see that
$$\|L(g)(a)\|_{\max}\le \|a\|_{\max}
$$
for every $a$ in $A^{\alg}(\bC\rtimes G)$.  Hence $L(g)$ 
continuously extends further to $A(\bC\rtimes G)$. 
We now use \eqref{wfqwefwefqwfqewvfvdfvdvdfvfqewf4r4}   and  a similar argument as in the proof  of Lemma \ref{gbiojwogwerewgweeereEEEer}.\ref{efiijfoiofjeojgoifewfewfqwef} to show that also $R(g)$ extends.

Consequently, the pair $(L(g),R(g))$ determines a multiplier $\pi_{\bC}(g)$ in $M(A(\bC\rtimes G))$ such that
  \begin{equation}\label{eqwfqwefewfwf42r1}\pi_{\bC}(g)a=L(g)(a)\ ,\quad a\pi_{\bC}(g)=R(g)(a)  \end{equation}
  for arbitrary $a$ in $A(\bC\rtimes G)$. 
Using the  formulas \eqref{wfqwefwefqwfqewfqewf4r4}   one   checks as in the proof  of Lemma \ref{gbiojwogwerewgweeereEEEer}.\ref{efiijfoiofjeojgoifewfewfqwef} that
$\pi_{\bC}(g)$ is unitary for every $g$ in $G$, and that the map $\pi_{\bC}:G\to U(M(A(\bC\rtimes G)))$ is a homomorphism of groups.

  Hence we have obtained the desired homomorphism 
$$\pi_{\bC}:G\to U(M(A(\bC\rtimes G)))\ .$$
Using $A(\iota_{\bC})([f])=[f,e]$   and the formulas \eqref{eqwfqwefewfwf42r1} and  \eqref{wfqwefwefqwfqewfqewf4r4} one easily verifies the relation
\eqref{werglknkeglergwegwgwrg}, i.e., that
$$\pi_{\bC}(g)A(\iota_{\bC})([f]) \pi_{\bC}(g^{-1})=A(\iota_{\bC})([gf])$$
for all morphisms $f$ in $\bC$ and $g$ in $G$.
\end{proof}

 By Lemma \ref{gbiojwogwerewgweeereEEEer}.\ref{qeojqoergefewfqwef}
 the covariant representation $( A(\iota_{\bC}) ,\pi_{\bC}   )$  determines a  morphism of $C^{*}$-algebras \begin{equation}\label{ewfqewfqewfqwefqwfeqwfqewfqewf}
\nu_{\bC}:A(\bC)\rtimes  G\to A(\bC\rtimes G)\ .
\end{equation}

 Recall that we  assume that $\bC$ is  in $\Fun(BG,\nCcat)$.
\begin{theorem} \label{qrioqwfewfewfewfqef}
The  morphism $\nu_{\bC}:A(\bC)\rtimes  G\to A(\bC\rtimes G)$ is an isomorphism.
\end{theorem}
\begin{proof}
We have a canonical functor $$\iota_{A(\bC)}\circ \rho_{\bC}:\bC\to A(\bC)\rtimes G\ ,$$ see \eqref{sfdvsq34gervv} (applied to $A(\bC)$ in place of $\bC$)  for $\iota_{A(\bC)}$ and \eqref{eqwfqewf124rwrefqfew} for $\rho_{\bC}$.
As seen in the proof of Lemma \ref{gbiojwogwerewgweeereEEEer}.\ref{efiijfoiofjeojgoifewfewfqwef}.
 (again applied to $A(\bC)$ in place of $\bC$), we furthermore have a homomorphism $$\nu: G\to U(M(A(\bC)\rtimes G))\ .$$ The pair $(\iota_{A(\bC)}\circ \rho_{\bC},\nu)$ is a covariant representation of $\bC$ on $A(\bC)\rtimes G$.
By Lemma   \ref{gbiojwogwerewgweeereEEEer}.\ref{qeojqoergefewfqwef}
 it induces a functor \begin{equation}\label{qwfwefqwefeqwfqvcdewewcwqec}
\bC\rtimes G\to A(\bC)\rtimes G
\end{equation}   sending $(f,g)$  to $([f],g)$ for all morphisms $f$ of $\bC$ and $g$ in $G$. One easily checks that for  morphisms $f:C\to C'$ and $f'':C''\to C'''$ of $\bC$ and $g,g'$ in $G$ we have
$$([f'],g')([f],g)=\left\{  \begin{array}{cc}
([gf'\circ f],g'g)&gC'=C''\\0&else
\end{array}
 \right. \ .$$
We now use the universal property stated in Lemma \ref{egioerjgoergwegwergwregrwe} in order to extend the functor \eqref{qwfwefqwefeqwfqvcdewewcwqec} to a  homomorphism of $C^{*}$-algebras 
$A(\bC\rtimes G)\to A(\bC)\rtimes G$.
One checks by a calculation with generators that this homomorphism is an inverse of 
  $\nu_{\bC}$.
\end{proof}

Let $\bC$ be in $\Fun(BG,\nCcat)$.

 \begin{kor}\label{wethgiowrthrthwtgregrweg}
 The morphism $\iota_{\bC}:\bC\to \bC\rtimes G$ is isometric.
 \end{kor}
\begin{proof}
We have a commuting square$$\xymatrix{\bC\ar[rr]^{\rho_{\bC}}\ar[d]^{\iota_{\bC}}&&A(\bC)\ar[d]^{\iota_{A(\bC)}}\\ \bC\rtimes G\ar[r]^{\rho_{\bC\rtimes G}}&A(\bC\rtimes G) &\ar[l]^{\nu_{\bC}}_{\cong}A(\bC)\rtimes G}\ .$$
Since $\rho_{\bC}$  and $\rho_{\bC\rtimes G}$ are isometric by Lemma \ref{ergtegwergwergwrg},
and $\iota_{A(\bC)}$ is isometric by Lemma \ref{regiuwergrerwegweg},
 this diagram implies the assertion.
\end{proof}

 \section{Colimits and Crossed products}\label{wgoijorgregwergwgrwreg}

In Definitions  \ref{wergiojwergregregwergerg} and \ref{feivoevvewfvfevfdsv} the crossed product  of a $\C$-linear $*$-category or a $C^{*}$-category with $G$-action with $G$     was introduced in an ad-hoc manner. The goal of the present section is to relate the crossed product with the formation of   colimits over the $G$-action in the respective large categories of small $\C$-linear $*$-categories or   small $C^{*}$-categories, see Proposition \ref{regiueqrhgiqwgewgqgfewrg}.  
In the unital case we have a well-developed homotopy theory of  $\C$-linear $*$-categories or   $C^{*}$-categories \cite{DellAmbrogio:2010aa}, \cite{startcats}. In this case the 
crossed product represents the homotopy $G$-orbits of the category. The technically precise formulation of this fact using the language of $\infty$-categories will be given in Theorem \ref{weiogwgwerrewgwregwreg}. 
 
\begin{rem}
The crossed product of a $*$-algebra over $\C$ or a $C^{*}$-algebra with $G$-action  with  $G$ is also classically considered as a sort of homotopy $G$-orbits of the algebra.  As far as we can see this can only made precise by considering these algebras as categories and forming the homotopy orbits in the realm of categories. 
Thereby it looks like an accident that the operation of taking  homotopy orbits preserves algebras. \hB
\end{rem}

 
 We start with describing an endofunctor   \begin{equation}\label{wervwervnwkjrvfvdfsv}
L: \Fun(BG,\nClincat)\to  \Fun(BG,\nClincat)\ ,
\end{equation} see Remark \ref{wegjwiergoergrefgwrefer} below for motivation.

  Let $\bC$ be in $\Fun(BG,\nClincat)$. The object $L(\bC) $ of $ \Fun(BG,\nClincat)$ has the following description.
  \begin{enumerate}
  \item objects: The set of objects of $L(\bC )$ is the set $\Ob(\bC)\times G$ with the  diagonal $G$-action 
 $h(C,g):=(hC,hg)$ for all $h$ in $G$ and objects $(C,g)$ of $L(\bC)$. 
  \item morphisms:  For two objects $(C,g),(C^{\prime},g')$ in $L(\bC)$   the $\C$-vector space of morphisms is defined by $$\Hom_{L(\bC)}((C,g),(C^{\prime},g^{\prime})) :=\Hom_{\bC }(C,C^{\prime})\ .$$ The element corresponding to  $f$  in $\Hom_{\bC }(C,C^{\prime})$  will be denoted by $(f,g\to g')$.   The group $G$ acts by
  $h(f,g\to g^{\prime}):=(hf,hg\to hg^{\prime})$ for all $h$ in $G$.
  \item If $(f',g'\to g'')$ is a second morphism in $L(\bC)$ with $f:C'\to C''$, then the composition is given by $$(f',g'\to g^{\prime\prime})\circ (f,g\to g'):=(f'\circ f,g\to g'')\ .$$
     \item The $*$-operation is given by $(f,g\to g')^{*}:=(f^{*},g'\to g)$.
  \end{enumerate}

  Let now $\phi:\bC\to \bC^{\prime}$ be a  morphism in $\Fun(BG,\nClincat)$. 
  Then the morphism $L(\phi):L(\bC)\to L(\bC^{\prime})$  has the following description:
  \begin{enumerate}
  \item objects: For an object $(C,g)$ in $L(\bC)$   we set $ L(\phi)(C,g):=(\phi(C),g)$.
  \item  morphisms: For a morphism $(f,g\to g')$ in $L(\bC)$ we set $L(\phi)(f,g\to g'):=(\phi(f),g\to g')$.
  \end{enumerate}

  The functor $L$  preserves $C^{*}$-categories and unitality and therefore induces endofunctors on the corresponding subcategories 
  $\Fun(BG,\Clincat)$, $\Fun(BG,\nCcat)$ and $\Fun(BG,\Ccat)$ of $\Fun(BG,\nClincat)$.

     \begin{rem}\label{wegjwiergoergrefgwrefer}
  The restriction of $L$ to 
 $\Fun(BG,\Clincat)$  or $\Fun(BG,\Ccat)$  is the cofibrant replacement functor considered in 
 \cite[Lemma 14.5]{startcats}
 for the projective model category structure on $\Fun(BG,\Clincat)$ or $\Fun(BG,\Ccat)$. This will be employed below. In the present non-unital case it is just an ad-hoc construction going into the formulation of Proposition \ref{regiueqrhgiqwgewgqgfewrg} below. \hB
   \end{rem}

 
 For $\bD$ in $\nClincat $ we let $\underline{\bD}$ denote the object of $\Fun(BG,\nClincat)$ given by $\bD$ with the trivial $G$-action. 
 We have a canonical  morphism  \begin{equation}\label{qrqfewfqwefqewf}
c_{\bC}^{\alg} : L(\bC)\to \underline{ \bC\rtimes^{\alg} G}
\end{equation}
  in $\Fun(BG,\nClincat)$ given as follows:
\begin{enumerate}
\item objects: The functor $c_{\bC}^{\alg}$ sends the object $(C,g)$ of $L(\bC)$  to the object $g^{-1}C$ of $\bC$.
\item morphisms: The functor $c_{\bC}^{\alg}$ sends the morphism $(f,g\to g')$ of $L(\bC)$  to the morphism $(g^{-1}f,g^{\prime,-1}g)$ of $\underline{ \bC\rtimes^{\alg} G}$, see Definition \ref{wergiojwergregregwergerg}.\ref{wtiogwerfgewrgwergwerg} for notation.
\end{enumerate}
 
 In the case of $C^{*}$-categories we consider the  morphism \begin{equation}\label{wdvqwecqwdcasdc}
c_{\bC}: L(\bC)\stackrel{c^{\alg}_{\bC}}{\to}  \underline{\bC\rtimes^{\alg}  G}\stackrel{\eqref{qwefkjiqkwefqwefqewf}}{\to} \underline{\bC\rtimes  G}
\end{equation}
 in $\Fun(BG,\nCcat)$.
 
 Recall Theorem \ref{riguhqwieufqewfeqfqewf} stating that the  categories $\nClincat$ and $\nCcat$ are cocomplete. 
 Hence for $\bC$ in $\Fun(BG,\nClincat)$ (resp.  $\Fun(BG,\nCcat)$)  the object $\colim_{BG}L(\bC)$ exists in $\nClincat$ (resp. $\nCcat$).  
 The following proposition states that this colimit is given by the crossed products.
 
 \begin{prop}\label{regiueqrhgiqwgewgqgfewrg}\mbox{}
 \begin{enumerate}
 \item\label{werthgwterwrg} If $\bC$ is in $\Fun(BG,\nClincat)$, then the morphism $c_{\bC}^{\alg}$ in \eqref{qrqfewfqwefqewf} presents $\bC\rtimes^{\alg} G$ as the colimit of $L(G)$ in $\nClincat$.
  \item\label{werthgwterwrg1}   If $\bC$ is in $\Fun(BG,\nCcat)$, then the morphism $c_{\bC}$ in \eqref{wdvqwecqwdcasdc} presents $\bC $ as the colimit of $L(\bC) $ in $\nCcat$.
\end{enumerate}
 \end{prop}
\begin{proof}
In order to show Assertion \ref{werthgwterwrg}. we must show that the map
 \begin{equation}\label{fwvevwefvwevrvwere}
 \Hom_{\nClincat}(\bC\rtimes^{\alg}  G,\bD) \to\Hom_{\Fun(BG,\nClincat)}(L(\bC) ,\underline{\bD})  \ , \end{equation}
 $$(\phi:\bC\rtimes^{\alg}  G\to \bD)\mapsto (L(\bC)\stackrel{c_{\bC}^{\alg}}{\to} \underline{\bC\rtimes^{\alg}G}\stackrel{\underline{\phi}}{\to} \underline{\bD})$$ is a bijection
for any $\bD$ in $\nClincat$.  To this end we describe the construction of the inverse of \eqref{fwvevwefvwevrvwere}.
Let $\phi:L(\bC)\to \underline{\bD}$  be in $\Hom_{\Fun(BG,\nClincat)}(L(\bC) ,\underline{\bD})$.  Then the inverse of  \eqref{fwvevwefvwevrvwere} sends $\phi$ to the functor     $\sigma^{\alg}:\bC\rtimes^{\alg}  G\to \bD$ given as follows.
\begin{enumerate}
\item objects: The functor $\sigma^{\alg}$ sends the object $C$  of $\bC\rtimes^{\alg}  G $ to the object $\phi(C,e)$ of $L(\bC)$.
\item morphisms: The functor $\sigma^{\alg}$ sends the morphism $(f,g)$ in $\bC\rtimes^{\alg}  G$ to the morphism 
$\phi(f,e\to g^{-1})$ in $L(\bC)$.
\end{enumerate}
It is straightforward to check that this construction provides an
inverse of \eqref{fwvevwefvwevrvwere}. This finishes the verification of Assertion \ref{werthgwterwrg}.
 
In order to show Assertion   \ref{werthgwterwrg1}.-  we argue similarly. 
We must show that    \begin{equation}\label{fwvevwefvwevrvwere1}
 \Hom_{\nCcat}(\bC\rtimes   G,\bD) \to\Hom_{\Fun(BG,\nCcat)}(L(\bC) ,\underline{\bD})  
\end{equation}$$(\phi:\bC\rtimes  G\to \bD)\mapsto (L(\bC)\stackrel{c_{\bC}}{\to} \underline{\bC\rtimes G}\stackrel{\underline{\phi}}{\to} \underline{\bD})$$ is a bijection 
for every $\bD$ in $\nCcat$.
For the inverse of \eqref{fwvevwefvwevrvwere1},
given $\phi :L(\bC)\to \underline{\bD}$  we first construct  $\sigma^{\alg}:
 \bC\rtimes^{\alg}  G\to \bD$ as above. By the universal property of the completion functor it uniquely extends to a functor 
$\sigma:\bC\rtimes G\to \bD$.
\end{proof}

We will use the language of $\infty$-categories\footnote{More precisely we mean $(\infty,1)$-categories.} \cite{htt},\cite{Cisinski:2017} in order to capture the homotopy theory of unital $\C$-linear $*$-categories or unital $C^{*}$-categories. Recall that morphisms in $\Clincat$ and $\Ccat$
are functors. Up to this point\footnote{with the exception of Remark \ref{rgiojerogergergw}} we have neglected their $2$-categorical structure, namely that we have the notion of  natural transformations between functors. A natural transformation   which is implemented by unitaries is called a 
unitary isomorphism    \cite[Eef. 2.4]{DellAmbrogio:2010aa}, \cite[Def. 5.1]{startcats}. A morphism itself is a unitary equivalence if it can be inverted up to unitary isomorphism \cite[Def. 5.2]{startcats}. We let $W_{\Clincat}$ and $W_{\Ccat}$ denote the (large) sets 
 of unitary  equivalences   in $\Clincat$ or $\Ccat$.
 
 If $\cC$ is any category with a set  of morphisms $W$, then we can form the $\infty$-category $\cC[W^{-1}]$\footnote{If we model $\infty$-categories by quasi-categories, then we should   more precisely write $\Nerve(\cC)[W^{-1}]$, where $\Nerve(\bC)$ is the nerve of $\cC$} called the Dwyer-Kan localization \cite[Def. 1.3.4.15]{HA}, \cite[7.1.2]{Cisinski:2017}. The Dwyer-Kan localization comes with a canonical functor $\ell:\cC\to \cC[W^{-1}]$ which is universal for functors from $\cC$ to $\infty$-categories which send the morphisms in $W$ to equivalences.

\begin{ddd}[{\cite[Def. 5.7]{startcats}}]\label{erkgowergerwgwergergwerg}
We define the $\infty$-categories
$$ \Clincat_{\infty}:=\Clincat[W_{\Clincat}^{-1}] \ , \quad \Ccat_{\infty}:=\Ccat[W^{-1}_{\Ccat}]\ . $$
\end{ddd}

\begin{rem}
 In the reference \cite{startcats} we used a slightly different notation. Since the main emphasis there was put on  the $\infty$-categories they were denoted by the shorter symbols $\Clincat$ and $\Ccat$ while the ordinary categories  were denoted by the symbols $\Clincat_{1}$ and $\Ccat_{1}$.   \hB \end{rem}

Let
$$\ell^{\alg}:\Clincat\to \Clincat_{\infty}\ , \quad \ell:\Ccat\to \Ccat_{\infty}$$
denote the localization functors.
\begin{ddd}\label{wtihowergergwegergwreg}
A morphism in $\Fun(BG,\Clincat)$ or $\Fun(BG,\Ccat)$ is called a weak equivalence if it becomes a unitary equivalence after forgetting the $G$-action.
\end{ddd}

\begin{rem}
Note that for a weak equivalence between $G$-categories we do not require the existence of an equivariant inverse up to unitary equivalence.  But by Remark \ref{ewroigjwegergwerg9} below there is always a weakly invariant inverse.

We let $W_{\Fun(BG,\Clincat)}$ and $W_{\Fun(BG,\Ccat)}$ denote the classes of weak equivalences in the respective categories. 
By  \cite[7.9.2]{Cisinski:2017} we have canonical equivalences \begin{equation}\label{fdsvndjkvnksdfvsdfvsdfvsfvsfdv}
\Fun(BG,\Clincat)[W^{-1}_{\Fun(BG,\Clincat)}]\simeq \Fun(BG, \Clincat_{\infty})
\end{equation}
 and
$$\Fun(BG,\Ccat)[W^{-1}_{\Fun(BG,\Ccat)}]\simeq \Fun(BG, \Ccat_{\infty})\ .$$
\hB
\end{rem}

We will use the notation (compare with \cite[(47)]{startcats})
$$\ell_{BG}^{\alg}:\Fun(BG,\Clincat)\to \Fun(BG,\Clincat_{\infty})$$ and  $$ \ell_{BG}:\Fun(BG,\Ccat)\to \Fun(BG,\Ccat_{\infty})$$ for the localization functors.

The following theorem generalizes \cite[Theorem 14.6]{startcats}
from categories with trivial actions to arbitrary actions.
\begin{theorem}\label{weiogwgwerrewgwregwreg}
\mbox{}
\begin{enumerate}
\item\label{ergiuergergwerg} For $\bC$ in $\Fun(BG,\Clincat)$ we have an equivalence
$$\colim_{BG}\ell_{BG}^{\alg} (\bC)\simeq \ell^{\alg}(\bC\rtimes^{\alg}G)\ .$$
\item For $\bC$ in $\Fun(BG,\Ccat)$ we have an equivalence
$$\colim_{BG}\ell_{BG} (\bC)\simeq \ell (\bC\rtimes G)\ .$$
\end{enumerate}
\end{theorem}
\begin{proof}
  Let $\bC$ be in $\Fun(BG,\nClincat)$. 
 We then have a natural transformation \begin{equation}\label{}
\lambda_{\bC}:L(\bC)\to \bC
\end{equation}  with the following description:
   \begin{enumerate}
 \item objects: The functor $ \lambda_{\bC}$ sends the object $(C,g)$ of $L(\bC)$ to the object $C$ of $\bC$.
 \item morphisms: The functor $\lambda_{\bC}$ sends the morphism $(f,g\to g')$  of $L(\bC)$  to the morphism  $f$ of $\bC$.
 \end{enumerate}

If $\bC$ is unital, then   the unital  functor $\lambda_{\bC}:L(\bC)\to  \bC$  is a weak equivalence in the sense of Definition \ref{wtihowergergwegergwreg}, see the proof of \cite[Lemma 14.5]{startcats}.

By \cite[Thm. 1.4]{startcats} and  \cite[Rem. 1.6]{startcats} the categories $\Clincat$  and $\Ccat$ have combinatorial model category structures whose weak equivalences are the unitary equivalences. These model category structures model 
 the corresponding $\infty$-categories.
As explained in \cite[Rem. 14.4]{startcats} the categories
$\Fun(BG,\Clincat)$ and $\Fun(BG,\Ccat)$ have projective model category structures whose   weak equivalences  are as in Definition \ref{wtihowergergwegergwreg}.

By \cite[Lemma 14.5]{startcats}
the transformation $\lambda:L\to \id$ of endofunctors of $\Fun(BG,\Clincat)$ or $\Fun(BG,\Ccat)$ is a cofibrant replacement functor in both cases. By \cite[Prop. 14.3]{startcats}
we have the equivalences 
$$\colim_{BG}\ell_{BG}^{\alg} (\bC)\simeq \ell^{\alg}( \colim_{BG}L (\bC)  )$$ for $\bC$ in $  \Fun(BG,\Clincat)$ 
and 
$$\colim_{BG}\ell_{BG}  (\bC)\simeq \ell ( \colim_{BG} L(\bC) )$$ if $\bC$ is in $\Fun(BG,\Ccat)$.
The assertions of the theorem now follow from Proposition \ref{regiueqrhgiqwgewgqgfewrg} and the fact the inclusion functors in \eqref{qewfoihqoiewfqwfqwefqewf} and \eqref{qewfoihqoiewfqwfqwefqewf1} are left-adjoints and therefore preserve colimits.
\end{proof}

If a morphism  $\bC\to \bD$  in $\Fun(BG,\Clincat)$ or $\Fun(BG,\Ccat)$ is a weak equivalence according to Definition \ref{wtihowergergwegergwreg},
then it  does not necessarily   have 
  an
%
equivariant inverse.  From this point of view the conclusion of the next proposition might seem surprising.

\begin{prop}\mbox{}\label{efiobgebgewrverbvewvbev}
\begin{enumerate}
\item \label{wthiorthrherherthetrhtrh}
If $ \bC\to \bD$  is a weak equivalence in $\Fun(BG,\Clincat)$, then the induced morphism
$\bC\rtimes^{\alg}G\to \bD\rtimes^{\alg}G$
is a unitary equivalence.
\item  \label{wthiorthrherherthetrhtrh1}
If $ \bC\to \bD$  is a weak equivalence in $\Fun(BG,\Ccat)$, then  the induced morphism
$\bC\rtimes G\to \bD\rtimes G $
is a unitary equivalence. \end{enumerate}
\end{prop}

\begin{proof}

We will give a short argument based on Theorem \ref{weiogwgwerrewgwregwreg}.  We leave it as an instructive excercise to provide a direct proof using Proposition \ref{obgfbgbgfbsfdgbsdfb9} and Remark \ref{ewroigjwegergwerg9} below.

We give the argument in the case \ref{wthiorthrherherthetrhtrh}. The case \ref{wthiorthrherherthetrhtrh1}. is analoguous. 
In view of the equivalence \eqref{fdsvndjkvnksdfvsdfvsdfvsfvsfdv}
the morphism  $$\ell_{BG}^{\alg}(\bC)\to \ell_{BG}^{\alg}(\bD)$$ is an equivalence   in $\Fun(BG,\Clincat_{\infty})$. It follows that   the morphism
$$\colim_{BG}\ell_{BG}^{\alg}(\bC)\to \colim_{BG}\ell_{BG}^{\alg}(\bD)$$ is an equivalence in $\Clincat_{\infty}$. 
By Theorem \ref{weiogwgwerrewgwregwreg} 
the morphism $$\ell^{\alg}(\bC\rtimes^{\alg}G)\to \ell^{\alg}(\bD\rtimes^{\alg} G)$$ is an equivalence in $\Clincat_{\infty}$. 
Since all objects in $\Clincat$ are fibrant and cofibrant this implies that $\bC\rtimes^{\alg}G\to \bD\rtimes^{\alg} G$ is a unitary equivalence.
\end{proof}

It was essentially obvious from the construction that the crossed product is a functor on $\Fun(BG,\nClincat)$ or
$\Fun(BG,\nCcat)$, see \eqref{vfdqeiuhfviudfv} and \eqref{efkjbjkdfviuqr3}. Note that morphisms in these categories are strictly $G$-invariant functors.
In the unital case  we have seen in Theorem \ref{weiogwgwerrewgwregwreg} that  the crossed product represents a  colimit  over $BG$  of a diagram in the $\infty$-category 
$\Clincat_{\infty}$ or $\Ccat_{\infty}$. This is the conceptual explanation for the fact that the crossed product is functorial for functors which only satisfy a weaker form of equivariance. 

Let $\bC, \bC'$ be in $\Fun(BG,\Clincat)$.
\begin{ddd} \mbox{}
\begin{enumerate}
\item
A  weakly equivariant functor from $\bC$ to $\bC^{\prime}$ is a pair
$(\phi,\rho)$ consisting of: \begin{enumerate} \item  a not necessarily equivariant functor  $\phi: \bC \to  \bC^{\prime} $ \item 
 and a family $\rho:=(\rho(g))_{g\in G}$ of unitary isomorphisms  of functors $\rho(g):\phi\to g^{-1}\phi g$ \end{enumerate} such that for all $g,h$ in $G$ we have
$$(g^{-1}\rho(h)g) \rho(g)=\rho(hg)\ .$$ 
\item A unitary natural transformation $\kappa:(\phi,\rho)\to (\phi',\rho')$  between weakly equivariant functors is a unitary natural transformation 
$\kappa:\phi\to \phi'$ such that  $g^{-1}\kappa g\circ \rho(g)=\rho'(g)\circ \kappa$ 
for every $g$ in $G$.
\end{enumerate}
\end{ddd}
Note that weak equivariance is an additional structure on a morphism, not merely a property.
The similarity with Definition \ref{regiowerjhgwergwergwregrg}.\ref{qrgqfwwefqef} is not an accident, see also Remark \ref{rgiojerogergergw}.  

If $(\phi',\rho' ):\bC'\to \bC''$ is a second    weakly equivariant morphism, then the composition is the weakly equivariant morphism defined by
\begin{equation}\label{sdfbkjwobfvsfdvsfv}(\phi',\rho' )\circ (\phi,\rho):=(\phi'\circ \phi, \rho'\circ \rho)\ , \end{equation}   where
$(\rho'\circ \rho)(g):=(\rho'(g)\circ g^{-1}\phi g)\circ (\phi'\circ \rho(g))$.  \begin{ddd} We let $\widetilde\Fun(BG,\Clincat)$ be the following $2$-category:
\begin{enumerate} \item objects: The objects of $ \widetilde\Fun(BG,\Clincat)$  are the objects of 
$\Fun(BG,\Clincat)$.
\item morphisms: The $1$-morphisms are  the weakly equivariant functors.
\item $2$-morphisms: The $2$-morphisms are the unitary natural transformations between  weakly equivariant functors.
\item composition: The composition of $1$-morphisms  is given by \eqref{sdfbkjwobfvsfdvsfv}.
  \end{enumerate}
  \end{ddd}
  We have a canonical inclusion of a wide subcategory
  $$\Fun(BG,\Clincat)\to \widetilde \Fun(BG,\Clincat)$$
  which is the identity on objects and sends the equivariant functor  $\phi$ to the 
  weakly equivariant functor  $(\phi,\id)$, where $\id$ is the family consisting of identities of $\phi$. Note that it is    well-defined since by equivariance $g^{-1}\phi g=\phi$ for all $g$ in $G$.
  We let $\widetilde\Fun(BG,\Ccat)$ denote the full $2$-subcategory of $\widetilde\Fun(BG,\Clincat)$ consisting of $C^{*}$-categories. 
 We let  $ \Clincat_{2,1}$ and $\Ccat_{2,1}$ denote the $2$-categories obtained from $\Clincat$ and $\Ccat$ by 
 adding unitary natural transformations as $2$-morphisms.
  \begin{prop}\label{obgfbgbgfbsfdgbsdfb9}\mbox{}
  \begin{enumerate}
  \item \label{rhlhjkblfvbsfdbewrsdbs}
  The crossed product functor  \eqref{rqgkjbejkqveqvqcdscasdc} extends to a $2$-functor
  $$-\rtimes^{\alg}G:\widetilde\Fun(BG,\Clincat)\to \Clincat_{2,1}\ .$$
  \item\label{rhlhjkblfvbsfdbewrsdbs1} The crossed product functor  \eqref{efkjbjkdfviuqr3f} extends to a $2$-functor
  $$-\rtimes G:\widetilde\Fun(BG,\Ccat)\to \Ccat_{2,1}\ .$$
\end{enumerate}
  \end{prop}
\begin{proof}
We first show Assertion  \ref{rhlhjkblfvbsfdbewrsdbs}.
Assume that $\bC$ and $\bC^{\prime}$ are objects of $\Fun(BG,\Clincat)$ and that
 $(\phi,\rho):\bC\to \bC'$ is a weakly equivariant functor. Then we must define a functor
 \begin{equation}\label{fdrhiur9rq0} (\phi,\rho)\rtimes^{\alg} G:\bC\rtimes^{\alg}G\to \bC'\rtimes^{\alg}G \end{equation} in $\Clincat$ 
 in a functorial way. We consider the diagram
 $$\xymatrix{\bC\ar[r]^{\phi}\ar[d]^{\iota_{\bC}^{\alg}}&\bC'\ar[d]^{\iota_{\bC'}^{\alg}}\\ \bC\rtimes^{\alg}G\ar@{..>}[r]^{(\phi,\rho)\rtimes^{\alg} G}& \bC'\rtimes^{\alg}G}\ .$$
 Our plan is to apply Lemma \ref{gbiojwogwerewgwreer}.\ref{reogijwioegwergwgrw1} in order to construct the dotted arrow. 
 To this end we must extend the right-down composition to a covariant representation
 $(\iota_{\bC'}^{\alg}\circ \phi,\pi)$ of $\bC$ on $ \bC'\rtimes^{\alg}G $. The identity of $ \bC'\rtimes^{\alg}G$ corresponds by the same lemma to a covariant representation $(\iota_{\bC'}^{\alg},\mu)$  of $\bC^{\prime}$ on $ \bC'\rtimes^{\alg}G $, where   $\mu=(\mu(g))_{g\in G}$ is given in view of  \eqref{wefvbekjnvvjfrsfvfdv}  by $\mu(g)_{C'}=(\id_{C'},g)$  for every $g$ in $G$ and object $C'$ of $\bC'$. 
 We define for every object $C$ of $\bC$ and $g$ in $G$.
 $$\pi(g)_{C}:=\mu(g)_{g^{-1}\phi(gC)}\circ (\rho(g)_{C},e):\phi(C)\to \phi(gC)\ .$$ One checks that $\pi:=(\pi(g))_{g\in G}$ satisfies the Condition \ref{regiowerjhgwergwergwregrg}.\ref{qrgqfwwefqef}. This finishes the construction of $$(\phi,\rho)\rtimes^{\alg} G:\bC\rtimes^{\alg}G\to \bC'\rtimes^{\alg}G\ .$$
 The explicit description of the morphism $(\phi,\rho)\rtimes^{\alg}G$ is a follows:
 \begin{enumerate}
 \item objects: Its action on objects is given by the action of $\phi$.
 \item morphisms: It sends the morphism $(f,g)$ in $\bC\rtimes^{\alg} G$ with $f:C\to C^{\prime}$
  to $$\pi(g)_{C} (\phi(f),e)=(\rho(g)_{C'}\phi(f),g):\phi(C)\to g^{-1}\phi(gC')$$ (here we used \eqref{wfqwefwefwefqwfqrevevervewfqwef} and \eqref{wefvbekjnvvjfrsfvfdv}).
 \end{enumerate}
One checks in a straightforward manner that construction is compatible with the composition
 and the involution. 
 
 Assume now that $\kappa:(\phi,\rho)\to (\phi',\rho')$ is a unitary natural transformation between weakly equivariant functor.
 Then $\kappa$ gives rise to a unitary natural transformation of covariant representations
 $(\iota^{\alg}_{\bC'}\circ \phi,\pi)\to (\iota^{\alg}_{\bC'}\circ \phi',\pi')$.
 It in turn induces a unitary natural transformation
 $\kappa \rtimes G: (\phi,\rho)\rtimes G\to (\phi',\rho')\rtimes G$.
 This finishes the proof of Assertion \ref{rhlhjkblfvbsfdbewrsdbs}.
 
 In order to get   Assertion \ref{rhlhjkblfvbsfdbewrsdbs1} we postcompose the functor from
  Assertion \ref{rhlhjkblfvbsfdbewrsdbs} with the completion functor  \eqref{oidfhbuihsfuibfefefesdfbsfb} taking into account Lemma \ref{rgioqwfqwefqefe}.
   \end{proof}

  \begin{rem}\label{ewroigjwegergwerg9}
  Let $\bC$, $\bD$ be in $\Fun(BG,\Ccat)$ or $\Fun(BG,\Clincat)$. Let furthermore $(\phi,\rho):\bC\to \bD$ be a weakly equivariant functor
  such that the underlying functor $\phi:\bC\to \bD$ is a unitary equivalence (after forgetting the $G$-actions). 
  Then there exists a weakly equivariant functor  $(\psi,\lambda)$ and 
  unitary natural  isomorphisms $(\phi,\rho)\circ (\psi,\lambda)\cong (\id_{\bC},\id)$ and
  $  (\phi,\lambda)\circ (\phi,\rho)\cong (\id_{\bD},\id)$ of weakly equivariant functors. 
  In fact, if we choose a functor $\psi:\bD\to \bC$ (without any equivariance condition)
  and an isomorphism $\kappa:\phi\circ \psi\cong \id_{\bD}$, then there exists a unique choice for the cocycle $\lambda$ 
  such that $\kappa$ becomes a unitary natural  isomorphism
  $\kappa:(\phi,\rho)\circ (\psi,\lambda)\cong (\id_{\bC},\id)$  of weakly equivariant functors.
\hB  
   \end{rem}

\section{Exactness of crossed products}

The main results of the present section are    Theorem   \ref{fbgbgrbsfbsdfbs} and Theorem \ref{rhioohwhtwergergergwergweg} stating that the crossed product functor  preserves exact sequences and excisive squares.
On the way we show in Proposition \ref{qwefewfewfqewfc} that the functors $A^{\alg}$ and $A$
defined in Definitions \ref{hiuwegregwergweg} and \ref{zhioerhrthtrhehth}
  preserve exact sequences.

An exact sequence  in $\nClincat$ or $\nCcat$ is a sequence
$$\bC\stackrel{i}{\to} \bD\stackrel{\phi}{\to}  \bQ$$
of morphisms   which both induce bijections on the level of objects, and which induce  exact sequences on the level of morphisms spaces.  The morphism $\phi$ will be called a quotient morphism, and $i$ is the inclusion of an ideal. As $i$ is the inclusion of the kernel of $\phi$, and $\phi$ represents  the quotient of $\bD$ be the ideal $\bC$,
these two morphisms determine each other. Since   we do not  always  name the whole data 
of the exact sequence we will characterize quotient morphisms   and inclusions of ideals separately.

 \begin{ddd}\label{wekotegergwergwerg}\mbox{}
 \begin{enumerate}\item
  A morphism $ \phi:\bD\to \bQ$  in $\nClincat$ is a quotient morphism if it satisfies the following conditions:
   \begin{enumerate}
 \item The induced map $\phi:\Ob(\bD)\to \Ob(\bQ)$ is a  bijection  between the sets of objects. 
 \item\label{wergiowergwergegwergergergergrewgwreg} For every pair of objects $D,D^{\prime}$
 in $\bD$  the  induced map
 of $\C$-vector spaces $\Hom_{\bD}(D,D^{\prime})\to \Hom_{\bQ}(\phi(D),\phi(D^{\prime}))$ is surjective.
 \end{enumerate}
 \item A morphism $ \phi:\bD\to \bQ$  in $ \nCcat$ is a quotient morphism if it is one in $\nClincat$.\end{enumerate}
 \end{ddd}


\begin{ddd}\label{weigwioegerww}\mbox{}
 \begin{enumerate}\item
A morphism $i:\bC\to \bD$  in $\nClincat$ is the inclusion of an ideal if it satisfies the following conditions:

\begin{enumerate}
\item \label{werijogwegfrewwf0}The induced map $i:\Ob(\bC)\to \Ob(\bD)$ is a  bijection  between the sets of objects. \item  \label{werijogwegfrewwf1} The morphism $i$ induces   injective  maps on the level of morphism spaces.
\item\label{werijogwegfrewwf} If $f$ and $g$ are composable morphisms in $\bD$ such that $f$ or $g$ are in $\bC$, then $g\circ f$ belongs to $\bC$.
 \end{enumerate}
 \item A morphism $i:\bC\to \bD$  in $\nCcat$ is   the inclusion of an ideal if   it satisfies the following conditions:\begin{enumerate}
 \item $i$ is an ideal inclusion in $\nClincat$.
 \item For every two objects $C,C'$ in $\bC$ the subspace $i(\Hom_{\bC}(C,C'))$ in $\Hom_{\bD}(i(C),i(C'))$ is closed.
  \end{enumerate}
  \end{enumerate}
\end{ddd}


Above we  we gave  explicit ad-hoc definitions of quotients morphisms and ideal inclusions.
 Since we are interested how these notions interact with the formation of crossed products, i.e., with forming certain colimits,  it is useful to have characterizations in categorical terms.
  We consider a square 
   \begin{equation}\label{grgqefeqwfwerwerwerweqef}
 \xymatrix{\bC\ar[r]^{i}\ar[d]&\bD\ar[d]^{\phi}\\ 0[\Ob(\bQ)]\ar[r]&\bQ} 
\end{equation} 
  in $\nClincat$,
 where the lower horizontal map is the counit of the adjunction in \eqref{qfewfewfwefwefewfefwefwffq}.   
  \begin{lem}\label{wkgijowergerferfw}\mbox{}
  \begin{enumerate}
  \item \label{qruiefqwfweqfd} If the square in \eqref{grgqefeqwfwerwerwerweqef} is cartesian, then 
   $i$ is the inclusion of an ideal.
   \item \label{weigowergferfwef} If $\phi$ is in addition  a morphism in 
 $\nCcat$, then  $i$ is an inclusion of an ideal   in $ \nCcat$, and the square is cartesian in $ \nCcat$. \end{enumerate}
  \end{lem}
\begin{proof} We start with the verification of Assertion \ref{qruiefqwfweqfd}.
A cartesian square  in $\nClincat$ induces a cartesian  square on the level of
sets of objects. Since the lower horizontal morphism induces a bijection on objects, so does $i$. This verifies Condition \ref{weigwioegerww}.\ref{werijogwegfrewwf0}.
For any pair of objects $C,C'$ in $\bC$ we furthermore have an induced   cartesian square of morphism  sets
\begin{equation}\label{fvsdvsfdvqrfefefewe}
\xymatrix{\Hom_{\bC}(C,C')\ar[r]\ar[d]&\Hom_{\bD}(C,C')\ar[d]\\0 \ar[r]&\Hom_{\bQ}(C,C')}\ .
\end{equation} 
 Since the lower horizontal map is injective, we can conclude that the upper map is injective, too. Hence $i$ induces an injection on morphism sets. This is  Condition \ref{weigwioegerww}.\ref{werijogwegfrewwf1}.  We finally check Condition \ref{weigwioegerww}.\ref{werijogwegfrewwf} by a  straightforward calculation.
 
We now consider Assertion \ref{weigowergferfwef} and assume that $\phi$ is a morphism in $\nCcat$.
Then $\phi$ is continuous on morphism spaces.
 In view of the cartesian squares \eqref{fvsdvsfdvqrfefefewe} we conclude that the $i$ sends the morphism spaces of $\bC$ injectively to closed subspaces of the morphism spaces of $\bD$.  We can conclude that $i$ is an inclusion of an ideal in $\nCcat$. We furthermore see  that 
 $\bC$  is a $C^{*}$-category, witnessed by  the norm induced from the inclusion into $\bD$ via $i$. In particular, it is already  in $\npClincat$. 
Applying the $\op$-version of Proposition \ref{roijgeqroifqewqfewf}.\ref{ergergerwf} to the colocalization  
 in \eqref{wefqwefefefqwfqfeewfewfqewf} we can conclude that the square in \eqref{grgqefeqwfwerwerwerweqef} is cartesian in $\npClincat$. 
 Since $\nCcat\to \npClincat$ is the inclusion of a full subcategory  the square 
 is   also  cartesian in $\nCcat$. 
 \end{proof}

Note that   $\nClincat$ admits fibre products by Theorem \ref{riguhqwieufqewfeqfqewf}.  
Hence given a morphism $\phi:\bD\to \bQ$ in $\nClincat$ we can construct its kernel $i:\bC\to \bD$ by forming the pull-back \eqref{grgqefeqwfwerwerwerweqef}. Furthermore, if $\phi$ is a morphism in $\nCcat$, then by Lemma \ref{grgqefeqwfwerwerwerweqef}  its kernel automatically belongs to $\nCcat$.

 Let $\phi:\bD\to \bQ$ be a morphism in $\nClincat$ and form the pull-back \eqref{grgqefeqwfwerwerwerweqef}. 
  

\begin{lem}\label{woihbjowegeerwvfevwfv} \mbox{} \begin{enumerate} \item \label{qfrfeewfewfq} The following assertions are equivalent:
\begin{enumerate}
\item
$\phi:\bD\to \bQ$ is a quotient morphism. \item 
 $\phi$  is a bijection on the level of objects and the square in \eqref{grgqefeqwfwerwerwerweqef} is a  push-out square in $\nClincat$.
 \end{enumerate}
 \item \label{qfrfeewfewfq1}  If $\phi$ is a morphism in $\nCcat$,  then  following assertions are equivalent:
 \begin{enumerate}
\item
$\phi:\bD\to \bQ$ is a quotient morphism. \item 
 $\phi$  is a bijection on the level of objects and the square in \eqref{grgqefeqwfwerwerwerweqef} is  a push-out square in $\nCcat$. \end{enumerate}
  
 \end{enumerate}
\end{lem}
\begin{proof}
We start with  Assertion  \ref{qfrfeewfewfq}.
Assume that 
 $\phi:\bD\to \bQ$ is a quotient morphism. Then it is a bijection on the level of objects. We show that \eqref{grgqefeqwfwerwerwerweqef} is a push-out diagram by 
  checking the universal property. Let $\bT$ be in $\nClincat$. Assume that the bold part of the following commuting diagram is given
\begin{equation}\label{grgqefeqwfqefrgergerg}
 \xymatrix{\bC\ar[r]\ar[d]&\bD\ar[d]^{\phi}\ar[ddr]^{\theta}&\\ 0[\Ob(\bQ)]\ar[r]\ar[drr]&\bQ\ar@{..>}[dr]&\\&&\bT} \ .
\end{equation} 
We can define the dotted arrow as follows:
\begin{enumerate}
\item objects: On objects the dotted arrow is defined as $\theta\circ \phi^{-1}:\Ob(\bQ)\to \bT$.
\item morphisms: For objects $Q,Q^{\prime}$ of $\bQ$ we define a map   $$\Hom_{\bQ}(Q,Q^{\prime})\to \Hom_{\bT}(\theta(\phi^{-1}(Q)),\theta(\phi^{-1}(Q^{\prime}))$$  such that it sends a morphism $f$ in $
\Hom_{\bQ}(Q,Q^{\prime})$ to $\theta(\tilde f)$, where $\tilde f$ is any choice of a morphism in $\bD$ such that $\phi(\tilde f)=f$. Note $\tilde f$ exists  by condition \ref{wekotegergwergwerg}.\ref{wergiowergwergegwergergergergrewgwreg}.
\end{enumerate}
It is clear  from the cartesian square \eqref{fvsdvsfdvqrfefefewe} that the dotted arrow is well-defined and unique.

Assume now that $\phi$ induces a bijection on the level of objects, and that  \eqref{grgqefeqwfwerwerwerweqef} is a push-out diagram.  We have a factorization $$\bD\stackrel{\phi'}{\to} \bQ'\to \bQ$$ of $\phi$, where  $\bQ'$ in $\nClincat$ is a wide subcategory  of $\bQ$ given by the image of $\phi$ on the level of morphism spaces.  Using the universal property of the push-out we get the dotted arrow in \begin{equation}\label{grgqefeqwfqefrgerfewewfwegerg}
 \xymatrix{\bC\ar[r]\ar[d]&\bD\ar[d]^{\phi}\ar[ddr]^{\phi'}&\\ 0[\Ob(\bQ)]\ar[r]\ar[drr]&\bQ\ar@{..>}[dr]&\\&&\bQ'} \ .
\end{equation} 
Its existence implies that it is an isomorphism. Consequently $\phi$ is surjective on morphism  spaces and hence a quotient morphism.

We now show Assertion \ref{qfrfeewfewfq1}. Again assume first that  $\phi:\bD\to \bQ$ is a quotient morphism in $\nCcat$.  In view of Definition \ref{wekotegergwergwerg}  it is   a  quotient morphism in $\nClincat$.  By Assertion \ref{qfrfeewfewfq} the square \eqref{grgqefeqwfwerwerwerweqef} is a push-out in $\nClincat$.  By Lemma \ref{wkgijowergerferfw}.\ref{weigowergferfwef} it is a 
 cartesian square in $\nCcat$, so in particular a  commutative square in this category. Since  $\nCcat\to \nClincat$ is an   inclusion  of a 
 full subcategory   the square in 
 \eqref{grgqefeqwfwerwerwerweqef} is a push-out square in $\nCcat$.  Of course, $\phi$ is also bijective on objects.  
  
We now assume that $\phi$ is bijective in objects and that the square  in  \eqref{grgqefeqwfwerwerwerweqef} is a push-out diagram in $\nCcat$. We claim that it is then also a push-out square in $\nClincat$. Assuming the claim we can apply Assertion \ref{qfrfeewfewfq} and conclude that $\phi$ is a quotient map in $\Clincat$, hence also a quotient map in $\nCcat$.

In order to see the claim we form the push-out diagram
   \begin{equation}\label{grgqefeqwfwerwerwerweqef10}
 \xymatrix{\bC\ar[r]^{i}\ar[d]&\bD\ar[d]^{\phi}\\ 0[\Ob(\bQ)]\ar[r]&\bQ'} 
\end{equation}  in $\nClincat$. We then use the non-formal fact that $\bQ'$ is already a $C^{*}$-category. 
This fact  is witnessed by the norm on $\bQ'$ given by  $\|f\|_{\bQ'}:=\inf_{\tilde f\in \phi^{-1}(f)} \|\tilde f\|_{\bD}$, see e.g.  \cite[Cor. 4.8]{mitchc}. 
 Since   $\nCcat\to \nClincat$  is an inclusion of a full subcategory, the square is also
a push-out square in $\nCcat$. Consequently, 
  the canonical morphism $\bQ'\to \bQ$  determined by the universal property in $\nClincat$ of the push-out in \eqref{grgqefeqwfwerwerwerweqef10}    is an isomorphism. Hence  \eqref{grgqefeqwfwerwerwerweqef} is also a push-out square in $\nClincat$ as claimed.  
%
%
%
%
%
%
%
%
\end{proof}

%
 

%
%
%

  Let \begin{equation}\label{sdfvsdfvfsdvfvsdfvsdfvsfdv}
\bC\stackrel{i}{\to} \bD\stackrel{\phi}{\to}\bQ
\end{equation} 
  be a sequence of morphisms in $\nClincat$.
\begin{ddd} \label{ergiowejogergergwergwergwregw}\mbox{}\begin{enumerate}
\item\label{wthregwwerggewrgwerg} The sequence \eqref{sdfvsdfvfsdvfvsdfvsdfvsfdv}  is called exact if:\begin{enumerate}
\item $\phi:\bD\to \bQ$ is a quotient map.
\item $i$ fits into a cartesian square \eqref{grgqefeqwfwerwerwerweqef}.
\end{enumerate}
\item 
The sequence \eqref{sdfvsdfvfsdvfvsdfvsdfvsfdv} is called an exact sequence of $C^{*}$-categories, if:
\begin{enumerate}
\item $\bC,\bD,\bQ$ are $C^{*}$-categories.
\item The sequence is exact in the sense of \ref{wthregwwerggewrgwerg}.
\end{enumerate}
\end{enumerate}
\end{ddd}
We will use the notation  
 $$0\to \bC\stackrel{i}{\to}  \bD \stackrel{\phi}{\to} \bQ\to 0$$
in order to visualize exact sequences.

  A sequence in $\Fun(BG,\nClincat)$ of the shape \eqref{sdfvsdfvfsdvfvsdfvsdfvsfdv} will be called exact if it becomes an exact sequence after forgetting the $G$-action. Similarly, a sequence in 
  $\Fun(BG,\nCcat)$ will be called an exact sequence of $C^{*}$-categories, if it becomes an  exact sequence of $C^{*}$-categories after forgetting the $G$-action.

  \begin{theorem} \label{fbgbgrbsfbsdfbs}\mbox{}
  \begin{enumerate}
  \item \label{sivijaoddsvadsv} If  $$0\to \bC\to  \bD \to \bQ\to 0$$ is an exact sequence  in $\Fun(BG,\nClincat)$, then \begin{equation}\label{sfbsfbfbfbsfdbsdfb}
0\to \bC\rtimes^{\alg} G\to \bD\rtimes^{\alg} G\to \bQ\rtimes^{\alg} G\to 0
\end{equation}
 is an exact sequence in $\nCcat$.
  \item \label{giowejgoeregwegergrg}
 If  $$0\to \bC\to  \bD \to \bQ\to 0$$ is an exact sequence of $C^{*}$-categories  in $\Fun(BG,\nCcat)$ such that $\bD $ is  unital, then \begin{equation}\label{sfbsfbfbfbsfdbsdfb}
0\to \bC\rtimes G\to \bD\rtimes G\to \bQ\rtimes G\to 0
\end{equation}
 is an exact sequence  of $C^{*}$-categories   in $\nCcat$.\end{enumerate} 
\end{theorem}

Note that Assertion \ref{fbgbgrbsfbsdfbs}.\ref{sivijaoddsvadsv}  is obvious from the Definition \ref{wergiojwergregregwergerg} and the fact that a direct sum of  a family of exact sequences of $\C$-vector spaces is again an exact sequence  of $\C$-vector spaces.  So in the following we will concentrate on the case of $C^{*}$-categories.

\begin{rem}
In view of Proposition \ref{egwergergwrgw}
the Assertion \ref{fbgbgrbsfbsdfbs}.\ref{giowejgoeregwegergrg} contains as a special case the assertion that the $C^{*}$-algebraic crossed product described in Definition \ref{feivoevvewfvfevfdsv}  preserves exact sequences of $C^{*}$-algebras in which the middle algebra is unital. This is a well-known fact \cite[Prop. 3.19]{williams} , \cite[Prop. 2.4.8]{celg} which will be used in the proof of  \ref{fbgbgrbsfbsdfbs}.\ref{giowejgoeregwegergrg}. \hB
\end{rem}

  We start with the observation that colimits preserve quotient morphisms.   Let $\bI$ be a small category.
 A morphism in 
 $\Fun(\bI,\nCcat)$  is called a quotient morphism if its evaluation at every object of $\bI$ is a quotient morphism in the sense of Definition \ref{wekotegergwergwerg}.
  
 \begin{prop}\label{iuerwhgiuregwegergwerg}
 If $\phi:\bD\to \bQ$ is a quotient morphism in $\Fun(\bI,\nCcat)$, then $\colim_{\bI}\phi:\colim_{\bI} \bD\to \colim_{\bI}\bQ $ is a quotient morphism.
   \end{prop}


 \begin{proof}
 By Theorem \ref{riguhqwieufqewfeqfqewf} the pull-backs and push-outs considered below exist.
 Since colimits and limits in functor categories are formed object wise we have a pull-back and push-out square of the shape \eqref{grgqefeqwfwerwerwerweqef} in $\Fun(\bI,\nCcat)$. Applying $\colim_{\bI}$ and using that $0[\Ob(-)]$ preserves colimits (since it is the composition of two left-adjoints \eqref{1qfewfq} and \eqref{1qfevdfvdfvfdvfdwfq}) we get a push-out square in $\nCcat$
\begin{equation}\label{grgqefeqwfqeef}
 \xymatrix{\colim_{\bI}\bC\ar[r]^{!}\ar[d]&\colim_{\bI}\bD\ar[d]^{\colim_{\bI }\phi}\\ 0[\Ob(\colim_{\bI}\bQ)]\ar[r]&\colim_{\bI}\bQ} \ .
\end{equation} 
Furthermore, since $\Ob$ commutes with colimits and $\phi$ was a quotient map 
all morphisms are bijections on objects. It is not clear that the marked morphism is an ideal inclusion so that
this square might not be a pull-back square. But we  
can consider the pull-back square in $\nCcat$
\begin{equation}\label{grgqefeqwf3r43ferfqeef}
 \xymatrix{\bK\ar[r]\ar[d]&\colim_{\bI}\bD\ar[d]^{\colim_{\bI}\phi}\\ 0[\Ob(\colim_{\bI}\bQ)]\ar[r]&\colim_{\bI}\bQ} 
\end{equation} 
defining  the $C^{*}$-category $\bK$ (the kernel of $\colim_{\bI}\phi$). 
We have a  canonical morphism $\colim_{\bI}\bC\to \bK$.  
We use  the comparison with \eqref{grgqefeqwfqeef} in order 
to show by checking the  universal property  that \eqref{grgqefeqwf3r43ferfqeef} is still a push-out square.
This implies by Lemma \ref{woihbjowegeerwvfevwfv}.\ref{qfrfeewfewfq1} that $\colim_{\bI}\phi:\colim_{\bI}\bD\to \colim_{\bI}\bQ$ is a quotient map.  
\end{proof}

  An exact sequence of $C^{*}$-categories with a single object is an exact sequence of $C^{*}$-algebras in the usual sense. 
  
 Note that the morphisms in an exact sequence of $\C$-linear $*$-categories or $C^{*}$-categories belong to $\nClincatinj$ or  $\nCcatinj$ (i.e., they are injective, in fact bijective, on the level of objects) so that we can apply the functors $A^{\alg}$ or  $A$ from Definitions \ref{hiuwegregwergweg} and  \ref{zhioerhrthtrhehth}.
 
 \begin{prop} \label{qwefewfewfqewfc} \mbox{}
 \begin{enumerate}
 \item \label{ergieogrrgrergwergwreg} If  $$0\to \bC\to  \bD \to \bQ\to 0$$ is an exact sequence   in $\nClincat$, then \begin{equation}\label{ewqfqwkjuhqkjwefweqfeqf}
0\to A^{\alg}(\bC)\to A^{\alg}(\bD)\to A^{\alg}(\bQ)\to 0
\end{equation}
 is an exact sequence  in $\nAlgc$.
 \item \label{ergieogrrgrergwergwreg1}
 If  $$0\to \bC\to  \bD \to \bQ\to 0$$ is an exact sequence of $C^{*}$-categories, then  \begin{equation}\label{adfafadfadsfdasfadsfdasf}
0\to A(\bC)\to A(\bD)\to A(\bQ)\to 0
\end{equation}
 is an exact sequence  in $\nCalg$. \end{enumerate}
 \end{prop}
 \begin{proof}
 Assertion \ref{ergieogrrgrergwergwreg} is an immediate consequence of   Definition  \ref{hiuwegregwergweg} since direct sums of exact sequences of vector spaces are exact. Therefore we concentrate on the case of $C^{*}$-categories. In this case the assertion has been shown in  \cite[Lemma 8.68]{buen}. For the sake of completeness and since this result is a crucial ingredient of  the proof of the main Theorem  \ref{fbgbgrbsfbsdfbs} of the present paper we provide the argument.
 
By Assertion  \ref{ergieogrrgrergwergwreg}.   we have an exact sequence \eqref{ewqfqwkjuhqkjwefweqfeqf}, or equivalently,  a push-out and pull-back diagram
$$\xymatrix{A^{\alg}(\bC)\ar[r]\ar[d]&A^{\alg}(\bD)\ar[d]\\0[*] \ar[r]&A^{\alg}(\bQ)}$$
in $\nClincat$. Using the fact that all corners belong to $\npClincat$ and the $\op$-version of Proposition  \ref{roijgeqroifqewqfewf} applied to the    colocalization  
 in \eqref{wefqwefefefqwfqfeewfewfqewf}  we can conclude that this square is a   
push-out and pull-back  in $\npClincat$, too.
 The completion functor  from \eqref{oidfhbuihsfuibfefefesdfbsfb} (see also \eqref{ervevwevefvdsfvsdfvsdfv}) is a left-adjoint 
 and therefore preserves push-outs. Applying the completion functor to the square above 
   we get a push-out diagram   \begin{equation}\label{}
\xymatrix{A (\bC)\ar[r]\ar[d]&A (\bD)\ar[d]\\0 \ar[r]&A (\bQ)}
\end{equation}
in $\nCcat$.   
Since $A(\bC)\to A(\bD)$ is an isometric inclusion by Lemma \ref{ergkljewrogwergwregf}
 this square is also a pull-back square.  Hence the sequence in \eqref{adfafadfadsfdasfadsfdasf} is exact.
%
%
  \end{proof}
  
  Recall the functor $L$ from \eqref{wervwervnwkjrvfvdfsv}.
 \begin{lem} \label{qwefewfewffrfrfrfqewfc}
 If  $$0\to \bC\to  \bD \to \bQ\to 0$$ is an exact sequence in $\Fun(BG,\nClincat)$ (resp.  $\Fun(BG,\nCcat)$), then
 $$0\to L( \bC) \to   L(\bD)\to L(  \bQ) \to 0$$
 is an exact sequence in $\Fun(BG,\nCcat)$  (resp.  $\Fun(BG,\nCcat)$).
 \end{lem}
 \begin{proof}
 This is obvious from the definition of $L$ in  \eqref{wervwervnwkjrvfvdfsv}.
 \end{proof}

\begin{proof}[{Proof of Theorem \ref{fbgbgrbsfbsdfbs}.\ref{giowejgoeregwegergrg}}]
By Lemma \ref{qwefewfewffrfrfrfqewfc} we have an exact sequence
 $$0\to L(\bC) \to  L(\bD ) \to L(\bQ) \to 0$$
in $\Fun(BG,\nCcat)$ where in addition $L(\bD)\to L(\bQ)$ is a morphism in $\Fun(BG,\Ccat)$.
By  Lemma \ref{regiueqrhgiqwgewgqgfewrg}.\ref{werthgwterwrg1}  
the sequence  $$\bC\rtimes G\to \bD\rtimes G\to \bQ\rtimes G$$ is isomorphic to the sequence 
$$\colim_{BG} L(\bC) \to \colim_{BG} L(\bD) \to \colim_{BG} L(\bQ) \ .$$
By Lemma \ref{iuerwhgiuregwegergwerg} we conclude that
$  \bD \rtimes G\to  \bQ \rtimes   G$ is a quotient map. Furthermore,
since the inclusion functor in \eqref{qewfoihqoiewfqwfqwefqewf1} is a left-adjoint and therefore preserves colimits and we assume that $\bD$ is unital, the morphism  $  \bD \rtimes G\to  \bQ \rtimes   G$ is a morphism in $\Ccat$, i.e., a unital morphism between unital $C^{*}$-categories.\footnote{This can also be seen directly from the definition of the crossed product.}


Specializing \eqref{grgqefeqwfqeef} to $\bI=BG$ we have a push-out diagram 
\begin{equation}\label{grgqefeqwf3r43ferfeqwerqerqedfvdfvdfvdfvfdvvdfvrqeef}
 \xymatrix{\bC\rtimes G\ar[r]\ar[d]& \bD \rtimes G\ar[d]\\ 0[\Ob( \bQ\rtimes G)]\ar[r]& \bQ\rtimes G} 
\end{equation} 
in $ \nCcat$.
We now form the pull-back square  in $\nCcat$
\begin{equation}\label{grgqefeqwf3r4eefefef3ferfeqwerqerqerqeef}
 \xymatrix{\bK\ar[r]\ar[d]& \bD \rtimes G\ar[d]\\ 0[\Ob( \bQ\rtimes G)]\ar[r]& \bQ\rtimes G} 
\end{equation} 
defining the $C^{*}$-category $\bK$. 
We then have a natural morphism $j:\bC\rtimes G\to \bK$. It remains to show that this morphism is an isomorphism.

We first claim that $j$ is isometric. To this end we consider the commuting diagram
$$\xymatrix{&\bK\ar[d]&\\\bC\rtimes G\ar[ru]^{j}\ar[r]^{!!}\ar[d]^{\rho_{\bC\rtimes G}}& \bD\rtimes G\ar[d]^{\rho_{\bD\rtimes G}}\ar[r]&\bQ\rtimes G\ar[d]^{\rho_{\Q\rtimes G}}\\A(\bC\rtimes G)\ar[r]^{!} &A(\bD\rtimes G) \ar[r]&A(\bQ\rtimes G) \\A(\bC)\rtimes G\ar[r]\ar[u]_{\nu_{\bC}}^{\cong}&A(\bD)\rtimes G\ar[r]\ar[u]_{\nu_{\bD}}^{\cong}&A(\bQ)\rtimes G\ar[u]_{\nu_{\bQ}}^{\cong}}\ .$$
By Theorem \ref{qrioqwfewfewfewfqef}
 the lower vertical morphisms are isomorphisms as indicated.  The lower horizontal sequence is exact by the well-known exactness of the maximal crossed product for $C^{*}$-algebras  and the exactness of $A$ shown in Lemma \ref{qwefewfewfqewfc}.  In particular the morphism marked by $!$ is an isometric embedding. 
Since $\rho_{\bC\rtimes G}$ and $\rho_{\bD\rtimes G}$ are also isometric by Lemma \ref{ergtegwergwergwrg} we conclude that the morphism marked by $!!$  is an isometric embedding.  Since $\bK\to \bD\rtimes G$ is an isometric embedding by definition we conclude that $j$ is an isometric embedding.

In particular we can now define the  quotient $C^{*}$-category $ {\bD\rtimes G}/{\bC\rtimes G}$ fitting into the push-out
$$\xymatrix{\bC\rtimes G\ar[r]\ar[d]&\bD\rtimes G\ar[d]\\0[\Ob(\bD\rtimes G)]\ar[r]& {\bD\rtimes G}/{\bC\rtimes G}}\ .$$
We then have the bold part of the commuting diagram
$$\xymatrix{\bC\ar[rr]^{\iota_{\bC}}\ar[d]&&\bC\rtimes G \ar[d] \\\bD\ar@{-->}[dr]^{\rho}\ar[rr]^{\iota_{\bD}}\ar[d]&&\bD\rtimes G\ar[dl]\ar[d]\\ \bQ\ar@/_2cm/[rr]^{\iota_{\bQ}}\ar@{..>}[r]^-{\psi}& {\bD\rtimes G}/{\bC\rtimes G}\ar[r]^-{\kappa}&\ar@/^0.5cm/@{..>}[l]\bQ\rtimes G}\ .$$
We now use the assumption that $\bD $ is unital. Then   $\bD\rtimes G $ is also unital, and the  morphism $\bD\rtimes G\to  {\bD\rtimes G}/{\bC\rtimes G}$ is a unital morphism.   By Lemma \ref{qergiowegwregergewgrg}.\ref{trhiorghrethtrhrhetr} it provides  a covariant representation $(\rho,\pi)$ of $\bD$ on $ {\bD\rtimes G}/{\bC\rtimes G}$. In particular we have a unitary  natural transformation $\pi(g) :\rho\to g^{*}\rho$ for every $g$ in $G$. 

By a diagram chase we  obtain a factorization   $\psi:\bQ\to  {\bD\rtimes G}/{\bC\rtimes G}$ of $\rho$ such that
$\pi(g):\psi\to g^{*}\psi$  for every $g$ in $G$ (here we use that $\bD \to \bQ$ is a bijection on the  sets of objects).
By  Lemma \ref{qergiowegwregergewgrg}.\ref{oighgioergwergwerg} the covariant representation $(\psi,\pi)$ induces the morphism  
$\bQ\rtimes G\to  {\bD\rtimes G}/{\bC\rtimes G}$ which is necessarily an inverse to $\kappa$.

The fact that $\kappa$ is an isomorphism implies that $j$ is an isomorphism.
\end{proof}

  We consider a square \begin{equation}\label{asdv2e4fwdqwqev}
\xymatrix{\bA\ar[r]\ar[d]&\bB\ar[d]\\\bC\ar[r]&\bD}
\end{equation}
 in $\nCcat$.
  
  \begin{ddd}\label{weiogwegerewrgwreg}
  The square \eqref{asdv2e4fwdqwqev} is called excisive if:
  \begin{enumerate}
  \item The morphisms $\bA\to \bB$ and $\bC\to \bD$ are inclusions of ideals.  
  \item The quotients $\bB/\bA$ and $\bD/\bC$ are unital. \item \label{thgoijroihwthwhhteh} The induced morphism  $ \bB/\bA \to  \bD/\bC $  is  unital and a unitary equivalence.
  \end{enumerate}
  \end{ddd}
 
 \begin{rem}
 There is a topological $K$-theory functor for $C^{*}$-categories defined as the composition \begin{equation}\label{qrgijewofweqfqweqwefqefq}
K^{\nCcat}:\nCcat\stackrel{A^{f}, \eqref{ewvqwvwccxewcqwcxwd}}{\to} \nCalg\stackrel{K^{C^{*}}}{\to} \Sp\ ,
\end{equation}
 where $K^{C^{*}}$ is the topological $K$-theory functor for $C^{*}$-algebras. One motivation for Definition \ref{weiogwegerewrgwreg}  is the following:
 \begin{prop}\label{ergiuheigerwgqwfwefqwef}
 The functor $K^{\nCcat}$ sends excisive squares in $\nCcat$ to push-out squares in $\Sp$.
 \end{prop}
This proposition will be shown in \cite[Thm. 12.4]{cank}.  \hB
   \end{rem}

We consider a square of shape  \eqref{asdv2e4fwdqwqev}  in $\Fun(BG,\nCcat)$. 
It is called excisive if it is so after forgetting the $G$-action.
\begin{theorem} \label{rhioohwhtwergergergwergweg}
If \eqref{asdv2e4fwdqwqev} is an excisive square in
$\Fun(BG,\nCcat)$ such that $\bB $ and $ \bD$ are   unital, then
 \begin{equation}\label{asdv2e4fwdqwfefefeffqev}
\xymatrix{\bA\rtimes G\ar[r]\ar[d]&\bB\rtimes G\ar[d]\\\bC\rtimes G\ar[r]&\bD\rtimes G}
\end{equation}
is an excisive square in $\nCcat$.
\end{theorem}
\begin{proof}
The horizontal morphisms in \eqref{asdv2e4fwdqwfefefeffqev} are  ideal inclusions by  Theorem \ref{fbgbgrbsfbsdfbs}.\ref{giowejgoeregwegergrg}.
Furthermore, by the same theorem the morphism 
$ {\bB\rtimes G}/{\bA\rtimes G}\to {\bD\rtimes G}/{\bC\rtimes G}$ is isomorphic to the morphism
$({\bB}/{\bA})\rtimes G\to  ({\bD}/{\bC})\rtimes G$. The latter is a unitary equivalence by Assumption \ref{thgoijroihwthwhhteh} and  Proposition \ref{efiobgebgewrverbvewvbev}.
\end{proof}

\begin{rem}
We will use Theorem \ref{rhioohwhtwergergergwergweg} in  \cite{coarsek} in order to verify excisiveness of an equivariant coarse $K$-homology functor. The Theorem  \ref{rhioohwhtwergergergwergweg} was one of the initial motivations for the present paper.
\hB
\end{rem}

\bibliographystyle{alpha}
\bibliography{forschung}

\begin{thebibliography}{CELY17}

\bibitem[BEa]{cank}
U.~Bunke and A.~Engel.
\newblock {Additive $C^{*}$-categories and $K$-theory}.
\newblock \href{https://arxiv.org/abs/2010.14830}{arXiv:2010.14830}.

\bibitem[BEb]{coarsek}
U.~Bunke and A.~Engel.
\newblock {Topological equivariant coarse $K$-homology}.
\newblock \href{https://arxiv.org/abs/2011.13271}{arXiv:2011.13271}.

\bibitem[BE20]{buen}
U.~Bunke and A.~Engel.
\newblock {\em Homotopy theory with bornological coarse spaces}, volume 2269 of
  {\em Lecture Notes in Math.}
\newblock Springer, 2020.
\newblock \href{https://arxiv.org/abs/1607.03657}{arXiv:1607.03657}.

\bibitem[BR07]{bartels_reich}
A.~Bartels and H.~Reich.
\newblock {Coefficients for the Farrell--Jones conjecture}.
\newblock {\em Adv. Math.}, 209(1):337--362, 2007.

\bibitem[Bun19]{startcats}
U.~Bunke.
\newblock Homotopy theory with *categories.
\newblock {\em Theory Appl. Categ.}, 34(27):781--853, 2019.

\bibitem[Bus68]{busby}
R.C. Busby.
\newblock Double centralizeres and extensions of ${C^{*}}$-algebras.
\newblock {\em Transactions of the AMS}, 132(1):79--99, 1968.

\bibitem[CELY17]{celg}
J.~Cuntz, S.~Echterhoff, X.~Li, and G.~Yu.
\newblock {\em K-theory for group $C^{*}$-algebras and semigroup
  $C^{*}$-algebras}, volume~47 of {\em Oberwolfach Seminars}.
\newblock Birkh{\"a}user/Springer, 2017.

\bibitem[Cis19]{Cisinski:2017}
D.~C. Cisinski.
\newblock {\em Higher categories and homotopical algebra}, volume 180 of {\em
  Cambridge studies in advanced mathematics}.
\newblock Cambridge University Press, 2019.
\newblock Available online under
  \url{http://www.mathematik.uni-regensburg.de/cisinski/CatLR.pdf}.

\bibitem[Del10]{DellAmbrogio:2010aa}
I.~Dell'Ambrogio.
\newblock The unitary symmetric monoidal model category of small
  ${C^{*}}$-categories.
\newblock {\em Homology, Homotopy and Applications, Vol. 14 (2012), No. 2,
  pp.101-127}, 04 2010.

\bibitem[GLR85]{ghr}
P.~Ghez, R.~Lima, and J.E. Roberts.
\newblock ${W^{*}}$-categories.
\newblock {\em Pacific Journal of Mathematics}, 120:79--109, 1085.

\bibitem[Joa03]{joachimcat}
M.~Joachim.
\newblock {$K$-homology of $C^{\ast}$-categories and symmetric spectra
  representing $K$-homology}.
\newblock {\em Math. Ann.}, 327:641--670, 2003.

\bibitem[Lur]{HA}
J.~Lurie.
\newblock Higher {A}lgebra.
\newblock Available at
  \href{http://www.math.harvard.edu/~lurie/}{www.math.harvard.edu/lurie}.

\bibitem[Lur09]{htt}
J.~Lurie.
\newblock {\em Higher topos theory}, volume 170 of {\em Annals of Mathematics
  Studies}.
\newblock Princeton University Press, Princeton, NJ, 2009.

\bibitem[Mit02]{mitchc}
P.~D. Mitchener.
\newblock ${C^{*}}$-categories.
\newblock {\em Proceedings of the London Mathematical Society}, 84:375--404,
  2002.

\bibitem[Wil07]{williams}
D.~P. Williams.
\newblock {\em Crossed products of $C^{*}$-algebras}, volume 134 of {\em
  Mathematical Surveys and Monographs}.
\newblock AMS, 2007.

\end{thebibliography}

\end{document}